%% file: OSWa3.tex
\documentclass{amsproc}
\usepackage{amsmath}
\usepackage{amssymb}
\usepackage[mathcal]{eucal}
\usepackage[bb=pazo]{mathalfa}
\usepackage{rotating}
\usepackage{graphicx}
\usepackage[tableposition=below]{caption}
\usepackage{pigpen}
\usepackage{multirow}
\usepackage{bm}

\DeclareGraphicsExtensions{.png,.pdf,.eps}
\usepackage[all,2cell,cmtip]{xy}
\usepackage{tikz}
\usetikzlibrary{patterns,snakes}
\usepackage{color}
\usepackage{longtable}
\usepackage{soul}
\newtheorem{theorem}{Theorem}[section]
\newtheorem{lemma}[theorem]{Lemma}

\newtheorem{proposition}[theorem]{Proposition}

\theoremstyle{definition}

\newtheorem{remark}[theorem]{Remark}

\def\P{{\mathbb P}}
\def\Q{{\mathbb Q}}

\def\T{{\mathbb T}}
\def\Z{{\mathbb Z}}

\def\cB{{\mathcal B}}

\def\cD{{\mathcal D}}
\def\cE{{\mathcal E}}
\def\cF{{\mathcal F}}
\def\cG{{\mathcal{G}}}

\def\cL{{\mathcal L}}
\def\cM{{\mathcal M}}
\def\cN{{\mathcal{N}}}
\def\cO{{\mathcal{O}}}

\def\cS{{\mathcal S}}

\def\cU{{\mathcal U}}

\def\Q{{\mathbb{Q}}}

\def\DA{{\rm A}}
\def\DB{{\rm B}}

\def\DD{{\rm D}}
\def\DE{{\rm E}}
\def\DF{{\rm F}}

\def\fg{{\mathfrak g}}
\def\fh{{\mathfrak h}}
\def\fp{{\mathfrak p}}

\def\fso{{\mathfrak s \mathfrak o}}

\def\PGL{\operatorname{PGL}}

\def\SO{\operatorname{SO}}

\def\Spin{\operatorname{Spin}}

\def\HM{\cN}

\def\lra{\longrightarrow}

\def\ra{\rightarrow}
\def\lra{\longrightarrow}
\def\rat{\dashrightarrow}

\def\operatorname#1{\mathop{\rm #1}\nolimits}

\def\Aut{\operatorname{Aut}}

\def\Pic{\operatorname{Pic}}

\def\det{\operatorname{det}}

\def\rat{\operatorname{RatCurves^n}}

\newcommand{\ol}[1]{\overline{#1}}
\newcommand{\pb}{\ar@{}[dr]|(.25){\text{\pigpenfont J}}}
\newcommand{\pbd}{\ar@{}[d]|(.25){\text{\pigpenfont J}}}
\makeatletter
 \newcommand*\tl[1]{\mathpalette\wthelper{#1}}
\newcommand*\wthelper[2]{%
        \hbox{\dimen@\accentfontxheight#1%
                \accentfontxheight#11.15\dimen@
                $\m@th#1\widetilde{#2}$%
                \accentfontxheight#1\dimen@
        }%
}

\newcommand*\accentfontxheight[1]{%
        \fontdimen5\ifx#1\displaystyle
                \textfont
        \else\ifx#1\textstyle
                \textfont
        \else\ifx#1\scriptstyle
                \scriptfont
        \else
                \scriptscriptfont
        \fi\fi\fi3
}
\makeatother
\newcommand{\shse}[3]{0 ~\ra ~#1~ \lra ~#2~ \lra ~#3~ \ra~ 0}

\AtBeginDvi{}

\makeindex

\begin{document}

\title[Characterizing  the homogeneous variety  ${\DF_4(4)}$]{Characterizing  the homogeneous variety  $\bm{\DF_4(4)}$}

\author[Occhetta]{Gianluca Occhetta}
\address{Dipartimento di Matematica, Universit\`a di Trento, via
Sommarive 14 I-38123 Povo di Trento (TN), Italy} 
\email{gianluca.occhetta@unitn.it}
\thanks{First author supported by PRIN project ``Geometria delle variet\`a algebriche''. First and second author supported by the Department of Mathematics of the University of Trento. Third author  supported by JSPS KAKENHI Grant Number 17K14153.}
\thanks{}

\author[Sol\'a Conde]{Luis E. Sol\'a Conde}
\address{Dipartimento di Matematica, Universit\`a di Trento, via
Sommarive 14 I-38123 Povo di Trento (TN), Italy}
\email{lesolac@gmail.com}

\author[Watanabe]{Kiwamu Watanabe}
\address{Course of Mathematics, Programs in Mathematics, Electronics and Informatics,
Graduate School of Science and Engineering, Saitama University.
Shimo-Okubo 255, Sakura-ku Saitama-shi, 338-8570 Japan}
\email{kwatanab@rimath.saitama-u.ac.jp}

\begin{abstract}
In this paper we consider the $15$-dimensional homogeneous variety of Picard number one $\DF_4(4)$, and provide a characterization of it in terms of its varieties of minimal rational tangents.
\end{abstract}

\subjclass[2010]{Primary 14J45; Secondary 14E30, 14M15, 14M17}
\maketitle

\section{Introduction}

Among projective varieties with special properties, the Cartan variety is one of the most extraordinary. Representation theoretically, it is defined as the closed orbit of the minimal representation of the group $\DE_6$  
(see \cite[IV.1.3]{Z} or \cite[\S 1b]{LVdV}). Starting from the theory of division algebras, it is described as the Cayley plane (see \cite[Example 4.16, Section 6.2]{LM}, that is the projective plane over the complexified octonions. Within projective geometry, it is known as the fourth Severi variety:  
these are the only four smooth projective varieties $S\subset\P^{3\delta+2}$ of dimension $2\delta$, that can be isomorphically projected to $\P^{3\delta+1}$; the Cartan variety is the one corresponding to secant defect $\delta=8$ (\cite{Cha}, \cite[\S 4]{LVdV}, \cite[Chapter IV]{Z}). 
Finally, in the framework of Fano manifolds, it is the one of Picard number one whose VMRT at the general point is the Spinor variety $S_4$  in its natural embedding (the only $10$-dimensional variety of Picard number one and maximal dual defect).

Its general hyperplane section, which certainly inherits  some of its peculiarities, has not be so thoroughly studied. It is still a homogeneous variety, for the simple Lie group of type $\DF_4$, appearing as the closed orbit of the representation given by the fourth fundamental weight $\omega_4$, and it may be described as the set of traceless elements of the Cayley plane. 
In this paper we consider the problem of its characterization in terms of its VMRT.
 
Following the ideas of Hwang and Mok (see \cite{Hw}, \cite{Mk4} and the references therein), the VMRT at a general point $x\in X$ of a Fano manifold of Picard number one $X$, that is the set of tangent directions to minimal rational curves passing by $x$ on $X$, encodes substantial information about the geometry of $X$; in certain cases, such as most  rational homogeneous spaces of Picard number one (cf. \cite{HH}, \cite{Mk}), the VMRT at the general point $x$ completely determines $X$. In this setting, ``most'' means all but those determined by short exposed nodes of the corresponding Dynkin diagram, that is symplectic Grassmannians\footnote{In a private communication, Jun-Muk Hwang has informed us that this case has been recently completed, using differential geometric techniques, in a joint work with Qifeng Li.} and two varieties of type $\DF_4$, one of them being a hyperplane section of the Cartan variety. 
 
Recently we have shown that a different assumption on the family of minimal rational curves (namely, that the evaluation morphism is smooth and isotrivial, with the expected fibers $\cM_x$ over $x$) allows to characterize long-root rational homogeneous spaces with algebraic methods (\cite{OSW}), based on a characterization of flag manifolds in terms of rational curves (\cite{OSWW}, see also \cite[Theorem A.1]{OSW}).

Together with some projective geometrical constructions, the same method works in the case of the symplectic Grassmannians, if one further assumes that the projective embedding of each $\cM_x$ into $\P(\Omega_{X,x})$ is the expected one (see \cite{OSWa}).  
In a nutshell, it is the embedding of the VMRT-bundle in $\P(\Omega_X)$ that allows us to understand the way in which the VMRT and certain auxiliary subobjects coming from a detailed description of the projective geometry of the VMRT, get twisted when moved along rational curves. A similar approach can be applied for the characterization of $\DF_4(4)$ presented here.

As in the case of the symplectic Grassmannian, our methods allow us to single out a distinguished subfamily of minimal rational curves, whose VMRT is homogeneous at every point. Upon it we may then construct a smooth projective variety which dominates $X$, supporting as many independent $\P^1$-bundle structures as its Picard number. Then we may use \cite[Theorem A.1]{OSW} to claim that this variety is a flag manifold, so that $X$ will be homogeneous. 

The structure of the paper is the following: Section \ref{sec:prelim} contains some preliminary material and notation regarding rational homogeneous manifolds and flag bundles, as well as a detailed description of the family of lines on the variety $\DF_4(4)$, and of a codimension three special subfamily, which we call the family of isotropic lines.  In Section \ref{sec:prelimS4}  we gather the 
main results concerning the projective geometry of the spinor variety $S_4$ and of its general hyperplane sections.
In Section \ref{sec:relprojgeom} we show how to reconstruct, upon the VMRT,  certain ${\SO}_7$-bundles over the manifold $X$, and study some of their properties. They can be used (see Section \ref{sec:isotropic}) to prove that a certain subvariety of the family of rational curves $\cB_3 \subset \cU$ is indeed a subfamily.  This will allow us in Section \ref{sec:flags} to construct a $\DB_3$-bundle over $X$ of the appropriate type, which we prove to be the complete flag manifold of type $\DF_4$. 

\section{Preliminaries}\label{sec:prelim}

\subsection{Rational homogeneous manifolds: notation}

A rational homogeneous manifold is a projective variety obtained as a quotient $G/P$ of a semisimple algebraic group $G$; the subgroup $P$ is then said to be parabolic. It is well known that every parabolic subgroup of $G$ is determined, up to conjugation, by a set of nodes of the Dynkin diagram $\cD$ of $G$. In fact, considering a Cartan decomposition of $\fg$,
$$
\fg=\bigoplus_{\alpha\in\Phi_-}\fg_{\alpha}\oplus\fh\oplus \bigoplus_{\alpha\in\Phi_+}\fg_{\alpha}
$$
determined by a Cartan subalgebra $\fh\subset\fg$ and a base of positive simple roots $\Delta=\{\alpha_1,\dots,\alpha_n\}$,  
every set $J$ of nodes of $\cD$ determines a subset $\{\alpha_j,\,\,j\in J\}\subset\Delta$, and a parabolic subalgebra:
$$
\fp(J)=\bigoplus_{\alpha\in\Phi_-}\fg_{\alpha}\oplus\fh \oplus\bigoplus_{\alpha\in\Phi_+(J)}\fg_{\alpha},
$$
where $\Phi_+(J)$ is the set of positive roots satisfying that $\alpha-\alpha_j$ is not a root of $\fg$ for all $j\in J$. 

Note that, $\fp(\emptyset)=\fg$, and that the parabolic subalgebra associated with the full set of nodes of $\cD$ is a Borel subalgebra of $\fg$.

Then $\fp(J)$ determines a parabolic subgroup $P(J)\subset G$, and the corresponding rational homogeneous space $G/P(J)$ is represented by the Dynkin diagram $\cD$ marked at the nodes $J$. It makes then sense to write 
$$\cD(J):=G/P(J).$$

Furthermore, for any choice of two sets $J'\subset J$ of nodes of the Dynkin diagram, the natural morphism $\cD(J)\to\cD(J')$ is a smooth Mori contraction, equivariant with respect to the action of $G$,  
whose fibers are rational homogeneous manifolds, determined by the marked Dynkin diagram obtained from $\cD$ by removing the nodes in $J'$ and marking the nodes in $J\setminus J'$.

In the case in which $\fg$ is simple (that is, if the Dynkin diagram $\cD$ is connected), we will consider the nodes to be ordered as in \cite[Page 58]{Hum2}. 

Let us include here some of the rational homogeneous varieties that we will use along the paper, together with the way in which we will denote them, and their classical description:

$$
\renewcommand{\arraystretch}{1.3}
\begin{tabular}{|c|c|c|}
\hline
$\DB_3(1)$&$\Q^5$& smooth $5$-dimensional quadric \\ \hline \hline
$\DB_3(3)$&$\Q^6$   & family of planes in $\Q^5$  \\ \hline \hline
$\DD_4(1)$&$\Q^6$   & smooth $6$-dimensional quadric  \\ \hline \hline
$\DD_4(3)$&$S_3\cong\Q^6$   & $6$-dimensional spinor variety   \\ \hline \hline
$\DD_4(4)$&$S_3\cong\Q^6$   & $6$-dimensional spinor variety \\ \hline \hline
$\DD_5(1)$&$\Q^8$   &  smooth $8$-dimensional quadric \\ \hline \hline
$\DD_5(4)$&$S_4$  &  $10$-dimensional spinor variety \\ \hline \hline
$\DD_5(5)$&$S_4$  &  $10$-dimensional spinor variety \\ \hline \hline
$\DE_6(6)$&  $E^{16}$  & Cartan variety\\ \hline \hline
$\DF_4(4)$&    & smooth hyperplane section of $E^{16}$\\ \hline
\end{tabular}
$$\par
\medskip
Let us recall that, for $n\geq 3$, the spinor varieties, $\DD_n(n-1)$, $\DD_n(n)$ in our notation, are defined as the varieties parametrizing the two families of $(n-1)$-dimensional linear subspaces in the smooth quadric $\Q^{2n-2}\cong\DD_n(1)$.

\subsection{Flag bundles}\label{ssec:flags} 

We include here some notation regarding well-known facts on principal bundles and rational homogeneous bundles, that we will use later on.

Let $G$ be a semisimple algebraic group, and $Y$ be any smooth projective variety, we will denote by $H^1(Y,G)$ the set of degree $1$ cocycles in the analytic space associated with $Y$, that we denote by also by $Y$, with values in the group $G$. Any cocycle  $\theta\in H^1(Y,G)$ defines a {\it $G$-principal bundle} $\cE\to Y$, which is an analytic space constructed by glueing open sets $U_i\times G$, for a certain open covering $U_i\subset X$, by means of transitions $\theta_{ij}:U_{i}\cap U_{j}\to G$ (so that we glue $(u,g)\in U_{i}\cap U_{j}\to G$ with $(u,\theta_{ij}(u)g)\in U_{j}\cap U_{i}\to G$). It supports a natural holomorphic left action of $G$. 

Given any rational homogeneous manifold of the form $G/P$ ($P$ parabolic subgroup of $G$), we may consider the analytic space:
$$
\cE\times_G G/P,
$$
defined as the quotient of the product $\cE\times G/P$ by the relation $(e,gP)\cong (eh,h^{-1}gP)$, for all $h\in G$. It supports a natural holomorphic submersion $\pi_P:\cE\times_G G/P\to Y$, whose fibers are isomorphic to $G/P$. Note that the space $\cE\times_G G/P$ supports a relative ample line bundle over $X$, which is projective, hence it is a projective variety, and the morphism $\pi_P$ is projective. A projective variety constructed in this way is called a {\it $G/P$-bundle} over $X$. Moreover, given two parabolic subgroups $P'\subset P$, the natural map $G/P'\to G/P$ extends to a morphism $\cE\times_G G/P'\to \cE\times_G G/P$, 
commuting with the maps $\pi_P$, $\pi_{P'}$. 

In particular, considering $B$ a Borel subgroup of $G$, the cocycle $\theta$ defines a $G/B$-bundle over $Y$, $\pi_B:\cE\times_G G/B\to Y$, that dominates any other $G/P$-bundle constructed as above. We usually call $\pi_B$ the {\it $\cD$-bundle}, or the {\it flag bundle} associated with $\theta$, where $\cD$ denotes the Dynkin diagram of $G$. Moreover, for any minimal parabolic subgroup $P\subset G$ properly containing $B$ (corresponding to the marking of the Dynkin diagram $\cD$ at all the nodes but one), the associated morphism $\cE\times_G G/P\to \cE\times_G G/B$ is a $\P^1$-bundle.

Conversely, if $Y$ is simply connected and $\pi:Z\to Y$ is a smooth morphism of projective varieties whose fibers are isomorphic to a rational homogeneous space $G/P$, it follows that $Z$ is locally free over $X$, by a well-known theorem of Fischer and Grauert. The group $G'=\Aut^\circ(G/P)$ is known to be semisimple, and the transitions between the trivializations of $Z$ define a cocycle in $H^1(Y,G')$. As above, the cocycle $\theta$ determines a $G'$-principal bundle $\cE'$ over $Y$, such that $Z$ is isomorphic over $X$ to $\cE'\times_{G'}G/P$.

In the case $Y \simeq \P^1$, a $G$-principal bundle $\cE$, and its associated flag bundle, are determined, by Grothendieck's theorem (cf. \cite{gro1}), by a co-character of a Cartan subgroup of $G$, modulo the action of the Weyl group of $G$. Following \cite[Section 3]{OSW} this information may be  interpreted geometrically as follows: for every minimal parabolic subgroup $P_i\subset G$ properly containing $B$, the morphism $\cE\times_G G/B\to \cE\times_G G/P_i$ is a $\P^1$-bundle. Denoting by $K_i$ its relative canonical divisor, the equivalence class of the co-character defining $\cE$ is determined by the set of intersection numbers $K_i\cdot \Gamma$, where $\Gamma$ denotes a minimal section of the flag bundle over $\P^1$ (\cite[Proposition 3.17]{OSW}). Each of these parabolic subgroups $P_i$ corresponds to a node $i$ of the Dynkin diagram $\cD$, therefore we may represent the $\cD$-bundle by the {\it tagged Dynkin diagram} obtained by appending the intersection number $K_i\cdot \Gamma$ to the node $i$ (see \cite[Remark 3.18]{OSW}).

\subsection{Lines in the rational homogeneous manifold $\bm{ {\rm F}_4(4)}$.}

We introduce here the variety $\DF_4(4)$, whose characterization is the main goal of this paper, and some of its basic properties.

By abuse of notation, we denote by $\DF_4$ the unique algebraic simple Lie group with associated Dynkin diagram $\DF_4$:

$$
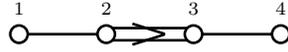
\input{DiagF4num}
$$ 

Numbering the diagram as above, the fourth node determines a maximal parabolic subgroup $P_4$, for which the homogeneous variety $\DF_4(4)=\DF_4/P_4$ is of Picard number one. As usual, we represent this variety by the corresponding marked Dynkin diagram: 
$$
\input{DiagF44}
$$

 The fourth node of the diagram is exposed short, in the language of \cite[Definition 2.10]{LM}, which in particular implies that the variety of minimal rational tangents (VMRT, for short) of the family of lines in $\DF_4(4)$ at every point is not homogeneous. 

On the other hand, it is known that the variety ${\rm F}_4(4)$ is a smooth hyperplane section of the (minimal degree embedding of the) homogeneous variety $\DE_6(6):={\rm E}_6/P_6$, corresponding to the marked Dynkin diagram:
$$
\input{DiagE66}
$$ 
(see, for instance, \cite[Section 6.3]{LM} and \cite[Chapter III. 2.5. F]{Z})

Hence the VMRT (at any point $x\in \DF_4(4)$) of the family of lines in $\DF_4(4)$ is a hyperplane section of  the VMRT of the family of lines in $\DE_6(6)$ which, in turn, can be described easily by means of marked Dynkin diagrams. Following \cite[Theorem 4.3]{LM}, this is the homogeneous variety determined by the marked Dynkin diagram obtained by deleting the marked node and marking its neighbors in the remaining diagram:
$$
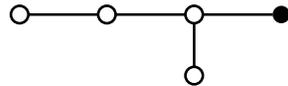
\input{DiagE66VMRT}
$$
In other words, this is the homogeneous variety classically known as the {\it spinor variety $S_4=\DD_5(5)$}. 
Furthermore, the tangent map $S_4\to \P(\Omega_{\DE_6(6),x})$ is the minimal embedding of this variety into a $15$-dimensional projective space $\P^{15}$. 

Note that $\dim(\DE_6(6))=16$, $\dim(\DD_5(5))=10$, and $\dim(\DF_4(4))=15$.
Being rational homogeneous, the variety $\DF_4(4)$ is a Fano variety, of index eleven, since its VMRT has dimension $9$. 
We have then a diagram as follows:

$$
\xymatrix@=42pt{\DE_6(6)&\DE_6(5,6)\ar[r]^{\P^1-\mbox{\small bdle.}}\ar[l]^q_{S_4-\mbox{\small bdle.}}&\DE_6(5)\\
\DF_4(4)\ar@{^{(}->}[]+<0ex,2.5ex>;[u]&q^{-1}(\DF_4(4))\ar@{^{(}->}[]+<0ex,2.5ex>;[u]%
\ar[r]^{\mbox{\small blowup}}\ar[l]^{q}&\DE_6(5)\ar@{=}[u]\\
\DF_4(4)\ar@{=}[u]&\cU\ar@{^{(}->}[]+<0ex,2.5ex>;[u]\ar[r]^{\P^1-\mbox{\small bdle.}}\ar[l]&\cM\ar@{^{(}->}[]+<0ex,2.5ex>;[u]\\
\DF_4(4)\ar@{=}[u]&\DF_4(3,4)\ar@{^{(}->}[]+<0ex,2.5ex>;[u]\ar[r]^{\P^1-\mbox{\small bdle.}}\ar[l]_{S_3-\mbox{\small bdle.}}&F_4(3)\ar@{^{(}->}[]+<0ex,2.5ex>;[u]}
$$
Here $\cU\to\cM$ denotes the universal family of lines in $F_4(4)$, so that the fibers of $\cU\to F_4(4)$ are smooth hyperplane sections of $S_4$, and $\cM$ may be seen as the set of zeroes of a section $s$ of the vector bundle $\cF$ on $E_6(5)$ satisfying that $\P(\cF)=E_6(5,6)$ and that $\cO_{\P(\cF)}(1)=q^*\cO_{E_6(6)}(1)$, corresponding to the section of $\cO_{E_6(6)}(1)$ defining $\DF_4(4)\subset \DE_6(6)$. The subvariety $\DF_4(3)\subset\cM$ parametrizes a subfamily of lines, called {\it isotropic lines of} $\DF_4(4)$ (since $\DF_4(3)$ is the closed orbit of the action of the group $\DF_4$ on $\cM$).

Let us finally denote by $U$ any smooth hyperplane section of $S_4$; they are all isomorphic, since $S_4\subset \P^{15}$ is self-dual (see \cite[Proposition 2.7]{Mu2}, \cite[Chapter III, Proposition 2.14]{Z}), and the automorphism group of $S_4$ acts transitively on $\P^{15}\setminus S_4$ (see \cite[Proposition 1.13]{Mu2}). Then all the fibers of the evaluation map $\cU\to \DF_4(4)$ are isomorphic to $U$. As we will state precisely in the next section, this property characterizes the variety $\DF_4(4)$.

\subsection{Statement of the main theorem}

Along the rest of the paper, unless otherwise stated, $X$ will denote a complex smooth Fano variety of Picard number one. We will also consider a dominating unsplit family of rational curves $p:\cU\to \cM$ in $X$, with an evaluation morphism $q:\cU\to X$. Equivalently, $\cM$ is a projective irreducible component of $\rat(X)$, $p$ is the restriction of the universal family of $\rat(X)$, and $q$ is the restriction of the evaluation morphism. As usual, we refer to \cite{kollar} for a complete account on the topic. We say that $\cM$ is {\it beautiful} if, moreover, $q$ is a smooth morphism. 

Denoting by $\P(\Omega_X)$ the Grothendieck projectivization of the cotangent bundle of $X$, there exists a rational map $\tau:\cU\dashrightarrow \P(\Omega_{X})$ associating to a general element $u\in \cU$ the tangent direction of the curve $p(u)$ at the point $q(u)$. Being $\cM$ dominant and unsplit, we may assert that, for general $x$, the composition of $\tau$ with the normalization map $\cM_x\to \cU$, that we denote by $\tau_x$, is a morphism (see \cite[Theorem 1]{HM2}, \cite[Theorem 3.4]{Ke2}). 

As in the previous section, we denote by $U\subset \P^{14}$ the minimal embedding of a smooth hyperplane section $U$ of (the minimal degree embedding of) the spinor variety $S_4$. We may now state the main result of this paper:

\begin{theorem}\label{thm:main4}
Let $X$ be a complex smooth Fano variety of Picard number one, and $\cM$ be a beautiful family of rational curves, whose tangent map $\tau$ is a morphism and satisfies that $\tau_x:\cM_x\to \P(\Omega_{X,x})$ is an embedding, projectively equivalent to $U\subset \P^{14}$, for every $x \in X$. Then $X$ is isomorphic to $\DF_4(4)$. 
\end{theorem}

\begin{remark}\label{rem:standard}  The assumptions of the above theorem imply (see \cite[Proposition~2.7]{Ar}) that every curve of the family $\cM$ is standard, i.e. the pullback of the tangent bundle of $X$ via the normalization of the curve
is isomorphic to $\cO_{\P^1}(2) \oplus \cO_{\P^1}(1)^{\oplus 9}\oplus \cO_{\P^1}^{\oplus 5}$. For short, we will say that the {\it splitting type} of $T_X$ on the curve is $(2,1^9, 0^5)$.
\end{remark}

\section{Projective geometry of  $S_4$ and its hyperplane sections}\label{sec:prelimS4}

We include here some facts on the geometry of the $10$-dimensional spinor variety $S_4$,  well known to the experts, which can be found scattered through the literature.

We start by considering an $8$-dimensional smooth quadric $\Q^8\subset \P^9$. The maximum dimension of a linear space contained in $\Q^8$ is four, and $\Q^8$ contains two irreducible families of $4$-dimensional linear spaces, parametrized by two isomorphic $10$-dimensional varieties, called the spinor varieties of $\Q^8$. The two types of $4$-dimensional spaces in $\Q^8$ will be denoted by $\P^4_+\in S^4_+$ and $\P^4_{-}\in S^4_{-}$. It is also known that the intersection of a $\P^4_+$ and a $\P^4_{-}$ is either empty or a linear space of odd dimension (see \cite[Theorem 22.14]{Harr}), and that the family of $\P^3$'s contained in $\Q^8$ is parametrized by the variety:
$$
\{(\P^4_+,\P^4_{-})|\,\,\dim(\P^4_+\cap \P^4_{-})=3\}.
$$
In other words, every $\P^3$ in $\Q^8$ is contained in precisely one $\P^4$ of each family.
On the other hand, the varieties described above can be considered as homogeneous varieties with respect to the algebraic group $\SO_{10}$, so the are determined by different markings of the Dynkin diagram $\DD_5$, 
$$
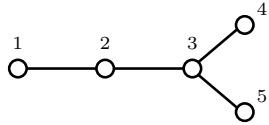
\input{DiagD5num}
$$
each of them corresponding to a parabolic subgroup (up to conjugation). For instance, $\Q^8$  is determined by the marking of $\DD_5$ at the first node, so we denote it by $\DD_5(1)$. In a similar way,  $S^4_+$, $S^4_-$, and the family of $\P^3$'s in $\Q^8$ shall be denoted by $\DD_5(4)$, $\DD_5(5)$, and $\DD_5(4,5)$ respectively (see \cite[Section 23.3]{FH}). They all fit in the following diagram of rational homogenous varieties and equivariant contractions:
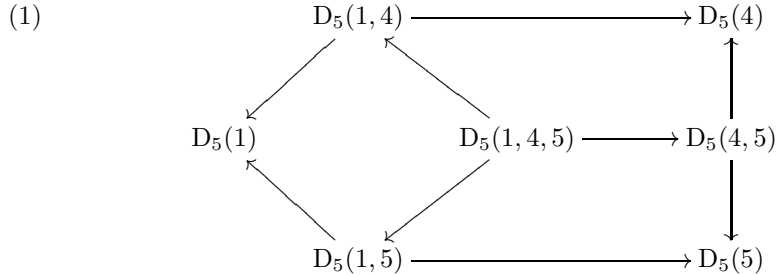
\begin{equation}
\xymatrix@R=30pt@C=15pt{&\DD_5(1,4)\ar[rrr]\ar[dl]&&&\DD_5(4)\\\DD_5(1)&&\DD_5(1,4,5)\ar[ul]\ar[dl]\ar[rr]&&\DD_5(4,5)\ar[u]\ar[d]\\&\DD_5(1,5)\ar[rrr]\ar[ul]&&&\DD_5(5)}
\end{equation}

If we fix a point $P\in\Q^8$, its inverse images into $\DD_5(1,4)$, $\DD_5(1,5)$, and $\DD_5(1,4,5)$ can be described as the rational homogeneous varieties $\DD_4(3)$, $\DD_4(4)$, and $\DD_4(3,4)$, respectively. Moreover, the upper and lower horizontal maps embed these $\DD_4(3)$ and $\DD_4(4)$ into $\DD_5(4)$ and $\DD_5(5)$; denoting by $j_+$ and $j_-$ the corresponding inclusions, and by  $B$ the fiber product of the maps $\DD_5(4,5)\to\DD_5(5)$ and $j_-:\DD_4(4)\to \DD_5(5)$, we have a commutative diagram:
\begin{equation}\label{eq:blowup}
\xymatrix@R=35pt@C=25pt{&\DD_4(3)\,\ar@{^{(}->}[rr]_{j_+}\ar[dl]&&\DD_5(4)\\\{P\}&\DD_4(3,4)\,\ar[u]\ar[d]\ar@{^{(}->}[r]&B\,\pbd\ar@{^{(}->}[r]\ar[dl]_{\pi}\ar[ur]^\sigma&\DD_5(4,5)\ar[u]\ar[d]\\&\DD_4(4)\,\ar@{^{(}->}[rr]^{j_-}\ar[ul]&&\DD_5(5)}
\end{equation}
where $\sigma$ is defined as the composition of the previously described maps $B\to\DD_5(4,5)\to \DD_5(4)$. 
Note that $j_+(\DD_4(3))$ (resp. $j_-(\DD_4(4))$) can be interpreted as the family of $\P^4_+$ (resp. $\P^4_-$) in $\Q^8$ passing by the point $P$. 
In other words, $B$ can be described as:
$$
B=\left\{(\P^4_+,\P^4_-)|\,\,\dim(\P^4_+\cap \P^4_{-})=3,\,\,\P^4_{-}\in j_-(\DD_4(4))\right\}.
$$

\begin{proposition}
The map $\sigma:B\to\DD_5(4)$ constructed above is the blow-up of $\DD_5(4)$ along $j_+(\DD_4(3))$.
\end{proposition}
\begin{proof} 
The union of all the lines in $\Q^8$ passing by $P$ is a $7$-dimensional quadric cone with vertex $P$; fix a smooth hyperplane section of this cone, and denote it by $\Q^6\subset\Q^8$; via intersection with $\Q^6$, the families $j_+(\DD_4(3))$ and $j_-(\DD_4(4))$ correspond to the two families of $\P^3$'s in $\Q^6$ (the $6$-dimensional spinor varieties). The elements of these families are denoted by $\P^3_+$, and $\P^3_-$, respectively.

 Let us denote by $O$ the open set $\DD_5(4)\setminus j_+(\DD_4(3))$ and by $\T_P\Q^8$ the embedded tangent space of $\Q^8$ at $P$. If $\P^4_+\in O$, the intersection of $\P^4_+$ and $\T_P\Q^8$ is a $3$-dimensional linear subspace $\P^3$. Then the join of $P$ and $\P^3$ is a $\P^4_-$ satisfying that $(\P^4_+,\P^4_-) \in B$. This means that $\sigma$ is surjective, since it is proper. 

Take  $\P^4_+\in O$ and an inverse image $(\P^4_+,\P^4_-) \in B$. By definition $\P^4_-$ contains $P$ and it is then contained in  $\T_P\Q^8$, whilst $\P^4_+$ is not contained in $\T_P\Q^8$. Then 
$\P^4_+\cap \T_P\Q^8$ is three-dimensional, and so $\P^4_+\cap \T_P\Q^8=\P^4_+\cap \P^4_-$.

This implies that $\P^4_-$ is the join of $P$ and $\P^4_+\cap \T_P\Q^8$, so it is uniquely determined by the choice of  $\P^4_+\in O$. In other words, $\sigma$ is an isomorphism over $O$.

On the other hand the image of $\DD_4(3,4)$ in $B$ is
$$\left\{(\P^4_+,\P^4_-)|\,\,\dim(\P^4_+\cap \P^4_{-})=3,\,\,\P^4_{+}\in j_+(\DD_4(3)),\, \,\P^4_{-}\in j_-(\DD_4(4))\right\},$$
that is, $\sigma^{-1}(j_+(\DD_4(3)))$. Then, using \cite[Theorem 1.1]{ESB}, one can prove that $\sigma$ is the smooth blow-up of $\DD_5(4)$ along $j_+(\DD_4(3))$.
\end{proof}

\begin{remark}\label{rem:linproj} The blowup $B$ of $\DD_5(4)$ can be interpreted geometrically as follows. Consider the natural embedding of $\DD_5(4)$ into $\P^{15}$. This is given by the $16$-dimensional spin representation of the group $\Spin_{10}$, whose projectivization has two orbits, with the closed one isomorphic to $\DD_5(4)$ (see, for instance \cite[Chapter III, Corollary 2.16]{Z}). Via this embedding, the subvarieties of the form $j_+(\DD_4(3))$ are smooth $6$-dimensional quadrics contained in $7$-dimensional linear subspaces of $\P^{15}$. Moreover, if we fix a $j_+(\DD_4(3))\subset \DD_5(4)$ as in Diagram (\ref{eq:blowup}), the rational map $\pi\circ\sigma^{-1}:\DD_5(4)\dashrightarrow\DD_4(4)$ corresponds to the linear projection of $\DD_5(4)$ from the linear span of $j_+(\DD_4(3))$:
\begin{equation}\label{eq:blowup3}
\xymatrix@R=35pt@C=25pt{\DD_4(3,4)\,\ar@{^{(}->}[r]\ar[d]&B\ar[d]_{\sigma}\ar[rrd]^{\pi}&&\\
\DD_4(3)\,\ar@{^{(}->}[r]^{j_+} \ar@{^{(}->}[]+<0ex,-2.5ex>;[d] &\DD_5(4)\ar@{-->}[rr]_{\pi\circ\sigma^{-1}} \ar@{^{(}->}[]+<0ex,-2.5ex>;[d]&&\DD_4(4) \ar@{^{(}->}[]+<0ex,-2.5ex>;[d]\\
\P^7\,\ar@{^{(}->}[r]&\P^{15}\ar@{-->}[rr]^{\mbox{\tiny{lin. proj.}}}&&\P^7}
\end{equation}
\end{remark}

The family of  $j_+(\DD_4(3))$'s in $\DD_5(4)$ provides a foliation in $\P^{15}$ with leaves $\P^7$'s. The precise statement and its dual (Proposition \ref{prop:uniqueQ6}), are a consequence of the following result, due to Ein and Shepherd-Barron:

\begin{proposition}{\cite[4.4]{ESB}}\label{prop:ESB} 
The linear system of quadrics containing $\DD_5(4)\subset\P^{15}$ provides a rational map $\P^{15}\dashrightarrow\Q^8$, that can be resolved via the blowup of $\P^{15}$ along $\DD_5(4)$, which is a $\P^7$-bundle over $\Q^8$. Moreover, the exceptional divisor of this blowup is isomorphic to $\DD_{5}(1,4)$, and the maps to $\DD_{5}(4)$ and to $\Q^8\cong\DD_5(1)$, are the natural equivariant maps of $\DD_5$-varieties.    
\end{proposition}

Using this result, one may prove the following:

\begin{proposition}\label{prop:uniqueQ6}
For every $P\in \P^{15}\setminus \DD_5(4)$ there exists a unique $j_+(\DD_4(3))\subset \DD_5(4)$ such that $P$ belongs to the $\P^7\subset\P^{15}$ spanned by it. For every hyperplane $H\subset \P^{15}$ non-tangent to $\DD_5(4)$, there exists a unique $j_+(\DD_4(3))\subset \DD_5(4)\cap H$.
\end{proposition}

\begin{proof} The first part follows from Proposition~\ref{prop:ESB} (see also \cite[Proposition 1.12]{Mu2}). The second is the dual statement of this, noting that the dual variety of 
$\DD_5(4)$ is $\DD_5(5)$ (see for instance \cite[Corollary 2.4]{Tev}, taking into account that the two half spin representations of $\fso_{10}$ are dual). In fact, given a hyperplane $H\subset\P^{15}$ non-tangent to $\DD_5(4)$, it corresponds to a point $P\in {\P^{15}}^\vee\setminus\DD_5(5)$. By the first part there exists a unique $7$-dimensional linear space $\P^7_{-}\subset{\P^{15}}^\vee$ containing $P$ and a smooth quadric $j_{-}(\DD_4(4))\subset\DD_5(5)$. 

Let $\rho:{\P^{15}}^\vee \dashrightarrow {\P^7_{+}}^{\!\vee}$ be the linear projection of ${\P^{15}}^\vee$ from  $\P^7_{-}$, so that, dually, we get a $7$-dimensional linear space $\P^7_{+}\subset\P^{15}$, that is contained in $H$.

Note that, arguing as in Remark \ref{rem:linproj} for the dual variety $\DD_5(5)$, the linear projection of $\DD_5(5)$ from $\P^{7}_{-}$ is a quadric $\DD_4(3) \subset {\P^7_{+}}^{\!\vee}$. Now, a hyperplane in $ {\P^7_{+}}^{\!\vee}$ tangent to $\DD_4(3)$ corresponds to a hyperplane in ${\P^{15}}^\vee$ containing $\P^7_{-}$ tangent to $\DD_5(5)$.
\end{proof}

\subsection{Projective geometry of the variety $U$}\label{ssec:projgeomU}

We will consider here a non-tangent hyperplane section $U=H\cap \DD_5(4)$ of $\DD_5(4)\subset \P^{15}$. 

The first thing we note  is that Proposition \ref{prop:uniqueQ6} implies that such a hyperplane section $U$ contains a unique $6$-dimensional quadric $j_+(\DD_4(3))$, and that the projection of $U$ from the linear span of this quadric maps $U$ onto a hyperplane section of $\DD_4(4)$, which is a smooth $5$-dimensional quadric $\DB_3(1)$. Moreover, the inverse image of $\DB_3(1)$ by $\pi$ in $B$ is the blowup of $U$ along $j_+(\DD_4(3))$, and its exceptional divisor $E:=\DD_4(3,4)\cap\pi^{-1}(\DB_3(1))$ is a $\P^2$-bundle over $\DD_4(3)=\DB_3(3)$. It can be interpreted as the universal family of $\P^2$'s contained in  $\DB_3(1)$, that is the homogeneous variety $\DB_3(1,3)$. Summing up, Diagram (\ref{eq:blowup3}) restricts to the following
\begin{equation}\label{eq:blowup4}
\xymatrix@R=35pt@C=25pt{\DB_3(1,3)\,\ar@{^{(}->}[r]\ar[d]&\pi^{-1}(\DB_3(1))\ar[d]_{\sigma}\ar[rrd]^{\pi}&&\\\DB_3(3)\ar@{^{(}->}[r] \ar@{^{(}->}[]+<0ex,-2.5ex>;[d]&U\ar@{-->}[rr]_{\pi\circ\sigma^{-1}} \ar@{^{(}->}[]+<0ex,-2.2ex>;[d]&&\DB_3(1) \ar@{^{(}->}[]+<0ex,-2.5ex>;[d]\\
\P^7\ar@{^{(}->}[r]&H\ar@{-->}[rr]^{\mbox{\tiny{lin. proj.}}}&&\P^6}
\end{equation}

\begin{remark}\label{rem:equivariant}
Note that the variety $U\subset\P^{14}$ is nondegenerate and linearly normal (since $\DD_5(4)$ is nondegenerate and linearly normal in $\P^{15}$), and $\Pic(U)\cong\Z$ hence every automorphism of $U$ 
extends to a projectivity of $\P^{14}$. In particular, any automorphism $g$ of $U$ sends the $6$-dimensional quadric $\DB_3(3)\subset\P^7$ into a $6$-dimensional quadric. Since there is only one such quadric contained in $U$, it follows that $g$ stabilizes $\DB_3(3)$ and its linear span $\P^7$ in $\P^{14}$. Then, on one hand $g$ extends to an automorphism of the blowup $\sigma^{-1}(U)$ (stabilizing $\DB_3(1,3)$), and, on the other, induces a projectivity of the projection $\P^6$ stabilizing $\DB_3(1)$. In other words, the group $\Aut(U)$ acts on all the varieties appearing in Diagram (\ref{eq:blowup4}), in a way in which the maps are equivariant. \end{remark}

\section{Construction of $\bm{{\SO}_7}$-bundles over $\bm{X}$}\label{sec:relprojgeom}

Let us first consider the evaluation morphism $q:\cU\to X$. Since, by hypothesis, $q$ is isotrivial, it is locally trivial (in the analytic topology) by a theorem of Fischer and Grauert, and so it is determined by a degree one cocycle $\theta$ on $X$ (consider here as a complex manifold) with values in the group $\Aut (U)$. Since moreover $X$ is simply connected, we may assume that $\theta$ lies in $H^1(X,\Aut^\circ(U))$. Furthermore, $\theta$ defines an $\Aut^\circ(U)$-principal bundle $\cE$ over $X$, that  we may use together with the action of $\Aut^\circ(U)$ on the varieties appearing in Diagram (\ref{eq:blowup4}) --see Remark \ref{rem:equivariant}--, in order to construct the following diagram of fiber bundles over $X$:

\begin{equation}\label{eq:diagram1}
\xymatrix@R=35pt@C=25pt{\cB_{13}\ar@{^{(}->}[r]\ar[d]&\tl{\cU}\ar[d]_{\sigma}\ar[rrd]^{\pi}&&\\\cB_3\ar@{^{(}->}[r] \ar@{^{(}->}[]+<0ex,-2.5ex>;[d]&\cU\ar@{-->}[rr]_{\pi\circ\sigma^{-1}} \ar@{^{(}->}[]+<0ex,-2.5ex>;[d]&&\cB_1 \ar@{^{(}->}[]+<0ex,-2.5ex>;[d]\\
\P(\cF_3)\,\ar@{^{(}->}[r]&\P(\Omega_X)\ar@{-->}[rr]&&\P(\cF_1)}
\end{equation}

Let us explain the notation introduced above: every variety is defined as the fiber space defined by the cocycle $\theta$ and the corresponding element of Diagram (\ref{eq:blowup4}). For instance, the fibers of $\tl{\cU}$ over points of $X$ are the blowups of the fibers of $\cU$ (isomorphic to the variety $U$) along the fiber of $\cB_3$, which is a smooth $6$-dimensional quadric. 

Note that the variety $\cU$ is fiberwise nondegenerate and linearly normal in $\P(\Omega_X)$, hence a trivialization of $\cU$ provides a trivialization of $\P(\Omega_X)$, and we may claim that $\P(\Omega_X)$ is precisely the projective bundle that one obtains by putting together the cocycle $\theta$ and the action of $\Aut^\circ(U)$ on the linear span $\P^{14}$. 

Now, the action on $\P^7$ provides a $\P^7$-subbundle of $\P(\Omega_X)$, and we may use the restriction of the tautological bundle of $\P(\Omega_X)$ to claim that this subbundle is in fact the projectivization of a rank $8$ vector bundle $\cF_3$, which is a quotient of the bundle $\Omega_X$. We will denote $\P(\cF_3)$ by $\cU_3$. Note that we are not claiming that $\cF_3$ is an $\Aut^\circ(U)$-vector bundle. Denote by $\cF_1$ the kernel of the surjection $\Omega_X\to\cF_3$:
\begin{equation}\label{eq:fundseq}
0 \to \cF_1 \lra \Omega_X\lra\cF_3 \to 0
\end{equation}

Since the $\P^6$-bundle obtained from $\theta$ and the action of $\Aut^\circ(U)$ on $\P^6$ is (fiberwise) the linear projection of $\P(\Omega_X)$ from $\P(\cF_3)$,  it must be equal to the projectivization of  $\cF_1$; we will denote by $q_1: \cU_1=\P(\cF_1)  \to X$ the natural projection.

On the other hand, as we noted in Section \ref{ssec:projgeomU}, the maps from $\DB_3(1,3)$ to $\DB_3(1)$, and to $\DB_3(3)$, appearing in Diagram (\ref{eq:blowup4}) are the natural morphism between these $\DB_3$ varieties, that is those of the universal family of planes in the $5$-dimensional quadric $\DB_3(1)$. The group $\Aut^\circ(U)$ acts on $\DB_3(1)$, providing a homomorphism of algebraic groups $\Aut^\circ(U)\to \Aut(\DB_3(1))=\SO_7$, and we may claim that the corresponding bundles $\cB_1$, $\cB_3$, $\cB_{13}$ are defined by the same cocycle, image of $\theta$ by the associated map $H^1(X,\Aut^\circ(U))\to H^1(X,\SO_7)$.

In particular we may consider $\P(\cF_1)$ and $\P(\cF_3)$ as projective representations of the group $\SO_7$. The first one is defined by the Grothendieck projectivization of the dual of the natural linear representation of $\SO_7$, denoted by $V(\omega_1)$. The vector bundle on $X$ defined by  $\theta$ and $V(\omega_1)^\vee$ will be a twist $\cF_1\otimes\cL_1^\vee$, for some $\cL_1\in\Pic(X)$. Since moreover $V(\omega_1)\cong V(\omega_1)^\vee$, it follows that
\begin{equation}\label{eq:L1}
\cF_1\cong \cF_1^\vee\otimes\cL_1^{\otimes 2}, \quad\mbox{ and }\quad\cL_1^{\otimes 7}\cong\det(\cF_1).
\end{equation}
 
The case of $\P(\cF_3)$ is slightly different, since the spin representation of the Lie algebra $\fso_7$, denoted by $V(\omega_3)$, is not an $\SO_7$-module, but only a $\Spin_7$-module. However, its second symmetric power is an $\SO_7$-module, whose dual defines a vector bundle of the form $S^2\cF_3\otimes\cL_3^\vee$, for some $\cL_3\in\Pic(X)$. Again, this bundle is self-dual, and we have:
\begin{equation}\label{eq:L3}
\cF_3\cong \cF_3^\vee\otimes\cL_3, \quad\mbox{ and }\quad\cL_3^{\otimes 4}\cong\det(\cF_3).
\end{equation}

\subsection{Restricting the $\bm{ \SO_7}$-bundles to lines}

Given now a curve $\ell$ of the family $\cM$, we may consider the image of $\theta$ by:
$$
H^1(X,\Aut^\circ(U))\to H^1(X,\SO_7)\to H^1(\ell,\SO_7),
$$ 
which defines an $\SO_7$-principal bundle over $\ell$. Following \cite[Section 3.3]{OSW}, this principal bundle is classified by a 
tagged Dynkin diagram of type $\DB_3$,
$$
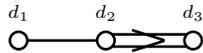
\input{DiagB3tagged.tex}
$$
 where $d_1,d_2,d_3$ are non-negative integers. 
 
 The pull-backs to $\ell$ of the bundles $\cF_1^\vee \otimes \cL_1$, $S^2\cF_3^\vee \otimes \cL_3$ are constructed upon this principal bundle and the $\SO_7$-modules $V(\omega_1)$, $S^2V(\omega_3)$, whose weights are well known (Cf \cite[19.3]{FH}). Then the tag of the $\DB_3$-bundle (understood as a co-character of a Cartan subgroup of $\SO_7$ \cite[Remark 3.7]{OSW}) allows us to compute the splitting type of these vector bundles (by using the pairing between characters and co-characters), and subsequently the tags of the corresponding $\PGL_7$ and $\PGL_8$ principal bundles on $\ell$. Note that these are the principal bundles defining the projective bundles obtained by pulling back  $\P(\cF_1)$, $\P(\cF_3)$ to $\ell$, that we denote by $\P(\cF_{1|\ell})$, $\P(\cF_{3|\ell})$. The result is the following:

\begin{lemma}\label{lem:tagB3}
With the same notation as above, the tagged Dynkin diagram associated with the $\PGL_7$-principal bundle defining the projective bundle $\P(\cF_{1|\ell})\to\ell$ is
$$
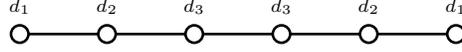
\input{DiagA6num1}
$$

On the other hand, the tagged Dynkin diagram associated with the $\PGL_8$-principal bundle defining the projective bundle $\P(\cF_{3|\ell})\to\ell$ is 
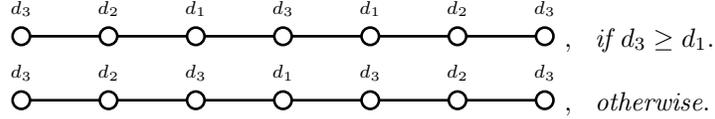
\[
  \begin{array}{l}
    \input{DiagA7num1}, ~\mbox{ if} ~ d_3 \geq d_1.\\
    \input{DiagA7num2}, ~\mbox{ otherwise}. 
    \end{array}
\]
\end{lemma}

\section{The family of isotropic lines in $X$}\label{sec:isotropic}

The main goal of this section is to show that 
the variety $\cB_3$ is a $\P^1$-bundle over its image in $\cM$, which will then be the natural candidate for the family of isotropic lines in $X$. This will follow from:

\begin{proposition}\label{prop:isotropic}
A fiber of $p:\cU \to \cM$ is either contained in ${\cB_{3}}$ or disjoint from it.
\end{proposition}

\begin{proof}
Let $\cO_{\cU}(1)$ be the restriction of the tautological line bundle $\cO_{\P(\Omega_{X})}(1)$ and $\cO_{\cU_i}(1)$ the tautological line bundles of $\cU_i$, satisfying ${q_1}_{\ast}(\cO_{\cU_1}(1))=\cF_1$ and ${q}_{\ast}(\cO_{\cU_3}(1))=\cF_3$.  Let us denote by $\cO_{{\tl{\cU}}}(1)$ the pull-back $\sigma^\ast \cO_{\cU}(1)$, by $\cO_{\cB_{13}}(1)$ the restriction $\cO_{{\tl{\cU}}}(1)|_{\cB_{13}}$ and by $\cO_{{\cB_{1}}}(1)$ the restriction $\cO_{\cU_1}(1)|_{\cB_{1}}$. Then we have 
\begin{equation}\label{eq:lin-proj}
\cO_{{\tl{\cU}}}(1)\otimes \cO_{{\tl{\cU}}}(-\cB_{13})\cong \pi^*\cO_{{\cB_{1}}}(1).
\end{equation}
Moreover, denoting  by $\cG$ and $\cS$ respectively the vector bundles  
$\pi_\ast(\cO_{{\tl{\cU}}}(1))$ and
$\pi_\ast(\cO_{\cB_{13}}(1))$  we have the following exact sequence on ${\cB_{1}}$:
\begin{equation}\label{eq:GS}
\shse{\cO_{{\cB_{1}}}(1)}{\cG}{\cS}.
\end{equation} 
The class of ${\cB_{1}}$
in $\Pic(\cU_1)$ is  given by $\cO_{\cU_1}(2)\otimes q_1^*F^{\vee}$ for some $F \in \Pic(X)$ and, by \cite[Proof of Proposition 4.10]{Fuhyp}, $\det(\cF_1)^{\otimes 2}=F^{\otimes 7}$, hence $F=\cL_1^{\otimes 2}$.
By the adjunction formula we have $K_{\cB_{1}}= (K_{\cU_1} + {\cB_{1}})|_{\cB_{1}}$, from which we get  
\begin{equation}\label{eq:K1}
\cO_{{\cB_{1}}}(K_{\cB_1/X}) = \cO_{\cB_1}(-5) \otimes (q_1^*\cL_1)^{\otimes 5}
\end{equation}
In the same way one can show that
\begin{equation}\label{eq:K3}
\cO_{{\cB_{3}}}(K_{\cB_3/X}) = \cO_{\cB_3}(-6) \otimes (q^*\cL_3)^{\otimes 3}
\end{equation}
Let $\overline \cB \to X$ be the  $\DB_3$-bundle associated with the $\DB_3(1,3)$-bundle $q \circ \sigma: \cB_{13} \to X$, and denote by $\rho$ the projection $\overline{\cB} \to \cB_{1,3}$.

The (pullbacks to $\overline{\cB}$ of the) relative canonicals of $q_1:\cB_1 \to X$, $q:\cB_3 \to X$ and $\pi:\cB_{13} \to \cB_1$ can be computed in terms of the relative canonicals of the $\P^1$-bundles of $\overline{\cB}$, as explained in
\cite[Lemma 2.3]{OSW}; the coefficients  $b_t$ and $c_t$ appearing in that statement can be found
in \cite[Table 1]{OSW}. In particular we get
\begin{eqnarray*}\label{eq:magic}
\rho^*\pi^*K_{\cB_1/X} & = & 5K_{\pi_1}+5K_{\pi_2}+5K_{\pi_3}\\
\rho^*\sigma^*K_{\cB_3/X} & = & 3K_{\pi_1}+6K_{\pi_2}+9K_{\pi_3}\\
\rho^*K_{\cB_{13}/\cB_1} & = & 2K_{\pi_2}+4K_{\pi_3}
\end{eqnarray*}
from which it follows that:
\begin{equation}\label{eq:relcan}
\dfrac{2}{3} \rho^*\sigma^*K_{\cB_3/X} - \dfrac{2}{5} \rho^*\pi^*K_{\cB_1/X} = \rho^*K_{\cB_{13}/\cB_1}
\end{equation}
Using the expressions for $K_{\cB_1/X}$ and $K_{\cB_3/X}$ given in (\ref{eq:K1}) and (\ref{eq:K3}) 
and writing $\cO_{{\cB_{13}}}(K_{\cB_{13}/\cB_1})$ as $\cO_{\cB_{13}}(-4) \otimes \pi^{\ast} \det(\cS)$ we obtain (using that $\rho^*, \pi^*, \sigma^*$ are injective maps of the Picard groups)
\begin{equation}\label{eq:c1S}
\det(\cS) = \cO_{\cB_1}(2) \otimes  (q_1^*\cL_3)^{\otimes 2} \otimes  (q_1^*\cL_1^\vee)^{\otimes 2}
\end{equation}
By the sequence (\ref{eq:GS}) we get that $\det(\cG)= \det(\cS) \otimes \cO_{\cB_1}(1)$, so we may
compute $\cO_{{\tl \cU}}(K_{\tl \cU/X})$ as:
\begin{eqnarray*}
\cO_{{\tl \cU}}(K_{\tl \cU/X}) & = & \cO_{{\tl \cU}}(K_{\tl \cU/\cB_1}) \otimes \pi^*\cO_{{\cB_1}}(K_{\cB_1/X}) \\
 & = & \cO_{\tl \cU}(-5) \otimes \pi^*(\det(\cG) \otimes \cO_{\cB_1}(-5) \otimes (q_1^*\cL_1)^{\otimes 5})\\
 &=& \cO_{\tl \cU}(-5) \otimes \pi^*( \cO_{\cB_1}(-2) \otimes (q^*\cL_1)^{\otimes 3}\otimes (q^*\cL_3)^{\otimes 2})
\end{eqnarray*}
Let now $\ell$ be a general fiber of $p:\cU \to \cM$ and denote by $\tl \ell$ its strict transform in $\tl\cU$; such a curve does not meet $\cB_3$, hence
$K_{\tl \cU} \cdot \tl\ell = K_\cU \cdot \ell = -2$. Moreover $\ell$ is a minimal section of $\P(\Omega_X)$ over its image in $X$, hence $\cO_{\P(\Omega_X)}(1) \cdot \ell =-2$, and therefore $\cO_{\tl \cU}(1) \cdot \tl\ell =-2$.
We now compute $K_{\tl \cU/X} \cdot \tl\ell$, using the expression obtained above, and we get
$$9=K_{\cU/X} \cdot \ell = K_{\tl \cU/X} \cdot \tl \ell = 14 +3\sigma^*{q}^*\cL_1 \cdot \tl\ell + 2\sigma^*q^*\cL_3 \cdot \tl\ell
= 14 +3{q}^*\cL_1 \cdot \ell + 2q^*\cL_3 \cdot \ell$$
On the other hand, the exact sequence (\ref{eq:fundseq}), pulled back to $\ell$, provides
$$7{q}^*\cL_1\cdot \ell +4 q^*\cL_3 \cdot \ell = -11.$$
The two conditions combined give 
\begin{equation}\label{eq:sarca}
{q}^*\cL_1\cdot \ell =q^*\cL_3\cdot \ell  =-1
\end{equation}
Since all fibers of $p$ are numerically equivalent to each other, equation (\ref{eq:sarca}) also holds for any fiber of $p$. 

From now on, let $\ell$ be any fiber of $p:\cU \to \cM$. 
From (\ref{eq:L1}) we have $\cF_1=\cF_1^\vee \otimes \cL_1^{\otimes 2}$; pulling back to $\ell$ we get
$\cF_{1|\ell}=\cF^\vee_{1|\ell}(-2)$; from the dual of the sequence (\ref{eq:fundseq}) we know that
$\cF^\vee_{1|\ell}$ is nef, therefore $\cF_{1|\ell}(2)$ is nef. By Lemma \ref{lem:tagB3}, combined with formulas (\ref{eq:L1}) and (\ref{eq:sarca}), we have 
\begin{equation}\label{eq:di's}
d_1+d_2+d_3=0~\mbox{or}~1.
\end{equation}

\medskip

From  Lemma \ref{lem:tagB3} it follows that the tag $(d_1, d_2, d_3)$ can not be $(0,0,0)$ and $(0,1,0)$ since, in that case, we will obtain that $q^*\det(\cF_3) \cdot \ell \equiv 0 \mod 8$, contradicting that $q^*\det(\cF_3) \cdot \ell = q^*\cL_3^{\otimes 4} = -4$ by formula (\ref{eq:sarca}).

Therefore the tag $(d_1, d_2, d_3)$  is $(1,0,0)$ or $(0,0,1)$, and, by Lemma \ref{lem:tagB3}, combined with formulas (\ref{eq:L1}), (\ref{eq:L3}) and (\ref{eq:sarca}) we have the following possibilities for the splitting types of $\cF_{1|\ell}$ and $\cF_{3|\ell}$:

$$
\renewcommand{\arraystretch}{1.5}
\begin{tabular}{|c|c|c|c|}
\hline
 Tag & $\cF_{1|\ell}$ & $\Omega_{X|\ell}$ & $\cF_{3|\ell}$  \\ \hline \hline
 $(1,0,0)$& $(-2,-1^5,0)$&$(-2,-1^9,0^5)$ & $(-1^4,0^4)$ \\ \hline
 $(0,0,1)$ &$(-2^3,-1,0^3)$ &$(-2,-1^9,0^5)$ &$(-2,-1^3,0^3,1)$\\ \hline
\end{tabular}
$$\par
\medskip
The inclusion of the minimal section $\ell$ in $\P(\Omega_X)$, corresponds to the unique quotient $\Omega_{X|\ell} \to \cO_\ell(-2)$. 

In the second case this quotient factors via $\cF_{3|\ell}$, so, equivalently, $\ell$ lies in $\cB_3$. 
In the first case the composition 
$$
\cF_{1|\ell}\hookrightarrow\Omega_{X|\ell} \lra \cO_\ell(-2)
$$
is surjective, which is to say that $\ell$ does not meet $\cB_3$. 
\end{proof}

\section{Construction of  relative flags and conclusion}\label{sec:flags}

We will denote now by $\cN$ the image of $\cB_3$ in $\cM$ via $p$.
As we have already remarked, $p:\cB_3 \to \cN$ is a $\P^1$-bundle, that we consider as a family of rational curves on $X$, whose evaluation map $q:\cB_3\to X$,  is a $\mathbb Q^6$-bundle.  
The proof of Theorem \ref{thm:main4} will be concluded by studying the associated $\DB_3$-bundle $\ol{\cB} \to X$, introduced in the proof of Proposition \ref{prop:isotropic}, and showing:

\begin{proposition}\label{prop:fine}
The variety $\ol{\cB}$ is a complete flag manifold.
\end{proposition}

Being a target of a contraction of $\ol{\cB}$, $X$ will then be a rational homogeneous manifold; since $\DF_4(4)$ is the only rational homogeneous manifold of Picard number one whose VMRT is isomorphic to $U$, Theorem \ref{thm:main4} will follow.

\begin{proof}[Proof of Proposition \ref{prop:fine}]
In order to prove that $\ol{\cB}$ is a complete flag, we will use \cite[Theorem A.1]{OSW}, that reduces it to show that $\ol{\cB}$ admits $\rho_{\ol{\cB}}=4$ independent $\P^1$-bundle structures. 
Note that $\ol{\cB}$ has already three different $\P^1$-bundle structures, coming from the flag bundle construction described in Section \ref{ssec:flags}. The last $\P^1$-bundle structure will be constructed by means of minimal sections over curves of $\cN$.

Let $\ell$ be a fiber of $p:\cB_3 \to \cN$   and consider, as in the previous section, the pull-back $\cB_{3|\ell}$, and the natural section $s:\ell\to s(\ell) \subset \cB_{3|\ell}$.  Denote by $ \ol{s}: \ell\to s^*\cB_{3|\ell}$ a minimal section of the bundle $s^*\ol\cB_{|\ell}$ over $\ell$ and  by $\ol{\ell}$ its image.  We have then the following commutative diagram:

$$
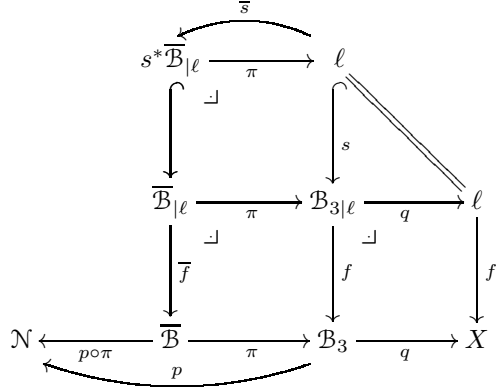
\xymatrix@=35pt{
& s^*\ol\cB_{|\ell}\pb \ar@{^{(}->}[]+<0ex,-2.5ex>;[d]\ar[r]_{\pi} 
\ar@{^{(}->}[]+<0ex,-2.5ex>;[d]
& \,\,\,\ell \ar@{=}[rd]  \ar@{^{(}->}[]+<0ex,-2.5ex>;[d]^s \ar@/_{2mm}/[]+<-2ex,+2.2ex>;[l]+<+0.7ex,+2.2ex>_{\ol{s}} &\\
& \ol\cB_{|\ell}\pb\ar[d]^{\ol f}\ar[r]_{\pi}&  \cB_{3|\ell}\pb\ar[d]^{f}\ar[r]_q& \ell\ar[d]^{f}\\
\HM & \ol{\cB}\ar[l]^{p\circ\pi}\ar[r]_{\pi}&\cB_3 
\ar@/^{3mm}/[]+<-2ex,-2ex>;[ll]+<+2ex,-2ex>_{p} \ar[r]_q&X }
$$

As in \cite[Section 4]{OSW} 
one can show that 
$\ol\ell$ is a minimal section of both the $\DB_3$-bundle $\ol{\cB}_{|\ell}$ and the $\DA_{2}$-bundle $s^*\ol{\cB}_{|\ell}$;  moreover  $s^*\ol{\cB}_{|\ell}$ is determined by the tagged Dynkin diagram obtained  by eliminating  the third node from the tagged Dynkin diagram of  $\ol{\cB}_{|\ell}$; this follows easily from the geometric interpretation of Grothendieck's theorem, given in Section \ref{ssec:flags}.
Since the tagged Dynkin diagram of   $\ol{\cB}_{|\ell}$ is (see the end of Proof of Proposition \ref{prop:isotropic})
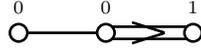
$$
\input{DiagB3tagged2.tex}
$$
then the tag of the $\DA_{2}$-bundle $s^*\ol{\cB}_{|\ell}$ is $(0,0)$.

In particular, by \cite[Proposition 3.17]{OSW}, the flag bundle $s^*\ol{\cB}_{|\ell}$ is trivial, and we may deduce, following verbatim \cite[Corollary 4.3]{OSW}, that there exists a smooth projective variety $\ol{\HM}$ such that the morphism $\ol{\cB}\to \HM$ factors via a smooth $\P^1$-bundle $\ol{p}:\ol{\cB}\to \ol{\HM}$. Now, following   \cite[Proof of Theorem 1.1]{OSW} we get that the four $\P^1$-bundle contractions of $\ol{\cB}$ are determined by  independent $K_X$-negative classes generating four extremal rays, and we conclude by \cite[Theorem A.1]{OSW}.
\end{proof}

\end{document}

%% file: DiagF4num.tex
\ifx\du\undefined
  \newlength{\du}
\fi
\setlength{\du}{3.3\unitlength}
\begin{tikzpicture}
\pgftransformxscale{1.000000}
\pgftransformyscale{1.000000}

%%%%%% COLORS
\definecolor{dialinecolor}{rgb}{0.000000, 0.000000, 0.000000} % EXTERIOR
\pgfsetstrokecolor{dialinecolor}
\definecolor{dialinecolor}{rgb}{0.000000, 0.000000, 0.000000} % INTERIOR
\pgfsetfillcolor{dialinecolor}

%%%%%% NODES

\pgfsetlinewidth{0.300000\du}
\pgfsetdash{}{0pt}
\pgfsetdash{}{0pt}
%\pgfsetmiterjoin

%%% #1
\pgfpathellipse{\pgfpoint{-6\du}{0\du}}{\pgfpoint{1\du}{0\du}}{\pgfpoint{0\du}{1\du}}
\pgfusepath{stroke}
\node at (-6\du,0\du){};
%\pgfpathellipse{\pgfpoint{-6\du}{0\du}}{\pgfpoint{1\du}{0\du}}{\pgfpoint{0\du}{1\du}}
%\pgfusepath{fill}
%\node at (-6\du,0\du){};

%%% #2
\pgfpathellipse{\pgfpoint{4\du}{0\du}}{\pgfpoint{1\du}{0\du}}{\pgfpoint{0\du}{1\du}}
\pgfusepath{stroke}
\node at (4\du,0\du){};
%\pgfpathellipse{\pgfpoint{4\du}{0\du}}{\pgfpoint{1\du}{0\du}}{\pgfpoint{0\du}{1\du}}
%\pgfusepath{fill}
%\node at (4\du,0\du){};

%%% #3
\pgfpathellipse{\pgfpoint{14\du}{0\du}}{\pgfpoint{1\du}{0\du}}{\pgfpoint{0\du}{1\du}}
\pgfusepath{stroke}
\node at (14\du,0\du){};
%\pgfpathellipse{\pgfpoint{14\du}{0\du}}{\pgfpoint{1\du}{0\du}}{\pgfpoint{0\du}{1\du}}
%\pgfusepath{fill}
%\node at (14\du,0\du){};

%%% #4
\pgfpathellipse{\pgfpoint{24\du}{0\du}}{\pgfpoint{1\du}{0\du}}{\pgfpoint{0\du}{1\du}}
\pgfusepath{stroke}
\node at (24\du,0\du){};
%\pgfpathellipse{\pgfpoint{24\du}{0\du}}{\pgfpoint{1\du}{0\du}}{\pgfpoint{0\du}{1\du}}
%\pgfusepath{fill}
%\node at (24\du,0\du){};

%%%%%% LINKS
\pgfsetlinewidth{0.300000\du}
\pgfsetdash{}{0pt}
\pgfsetdash{}{0pt}
\pgfsetbuttcap

{\draw (-5\du,0\du)--(3\du,0\du);}
%{\draw (5\du,0\du)--(13\du,0\du);}
{\draw (15\du,0\du)--(23\du,0\du);}
%{\draw (25\du,0\du)--(33\du,0\du);}
{\draw (4.65\du,0.7\du)--(13.35\du,0.7\du);}
{\draw (4.65\du,-0.7\du)--(13.35\du,-0.7\du);}

%%%%%% ARROW HEAD

{\pgfsetcornersarced{\pgfpoint{0.300000\du}{0.300000\du}}\definecolor{dialinecolor}{rgb}{0.000000, 0.000000, 0.000000}
\pgfsetstrokecolor{dialinecolor}
\draw (7\du,-1.2\du)--(10.8\du,0\du)--(7\du,1.2\du);}

%%%%%% TAGS
%\node[anchor=west] at (10\du,-10\du){${\rm F}_4$};

\node[anchor=south] at (-6\du,1.1\du){$\scriptstyle 1$};

\node[anchor=south] at (4\du,1.1\du){$\scriptstyle 2$};

\node[anchor=south] at (14\du,1.1\du){$\scriptstyle 3$};

\node[anchor=south] at (24\du,1.1\du){$\scriptstyle 4$};

\end{tikzpicture} 

%% file: DiagF44.tex
\ifx\du\undefined
  \newlength{\du}
\fi
\setlength{\du}{3.3\unitlength}
\begin{tikzpicture}
\pgftransformxscale{1.000000}
\pgftransformyscale{1.000000}

%%%%%% COLORS
\definecolor{dialinecolor}{rgb}{0.000000, 0.000000, 0.000000} % EXTERIOR
\pgfsetstrokecolor{dialinecolor}
\definecolor{dialinecolor}{rgb}{0.000000, 0.000000, 0.000000} % INTERIOR
\pgfsetfillcolor{dialinecolor}

%%%%%% NODES

\pgfsetlinewidth{0.300000\du}
\pgfsetdash{}{0pt}
\pgfsetdash{}{0pt}
%\pgfsetmiterjoin

%%% #1
\pgfpathellipse{\pgfpoint{-6\du}{0\du}}{\pgfpoint{1\du}{0\du}}{\pgfpoint{0\du}{1\du}}
\pgfusepath{stroke}
\node at (-6\du,0\du){};
%\pgfpathellipse{\pgfpoint{-6\du}{0\du}}{\pgfpoint{1\du}{0\du}}{\pgfpoint{0\du}{1\du}}
%\pgfusepath{fill}
%\node at (-6\du,0\du){};

%%% #2
\pgfpathellipse{\pgfpoint{4\du}{0\du}}{\pgfpoint{1\du}{0\du}}{\pgfpoint{0\du}{1\du}}
\pgfusepath{stroke}
\node at (4\du,0\du){};
%\pgfpathellipse{\pgfpoint{4\du}{0\du}}{\pgfpoint{1\du}{0\du}}{\pgfpoint{0\du}{1\du}}
%\pgfusepath{fill}
%\node at (4\du,0\du){};

%%% #3
\pgfpathellipse{\pgfpoint{14\du}{0\du}}{\pgfpoint{1\du}{0\du}}{\pgfpoint{0\du}{1\du}}
\pgfusepath{stroke}
\node at (14\du,0\du){};
%\pgfpathellipse{\pgfpoint{14\du}{0\du}}{\pgfpoint{1\du}{0\du}}{\pgfpoint{0\du}{1\du}}
%\pgfusepath{fill}
%\node at (14\du,0\du){};

%%% #4
%\pgfpathellipse{\pgfpoint{24\du}{0\du}}{\pgfpoint{1\du}{0\du}}{\pgfpoint{0\du}{1\du}}
%\pgfusepath{stroke}
%\node at (24\du,0\du){};
\pgfpathellipse{\pgfpoint{24\du}{0\du}}{\pgfpoint{1\du}{0\du}}{\pgfpoint{0\du}{1\du}}
\pgfusepath{fill}
\node at (24\du,0\du){};

%%%%%% LINKS
\pgfsetlinewidth{0.300000\du}
\pgfsetdash{}{0pt}
\pgfsetdash{}{0pt}
\pgfsetbuttcap

{\draw (-5\du,0\du)--(3\du,0\du);}
%{\draw (5\du,0\du)--(13\du,0\du);}
{\draw (15\du,0\du)--(23\du,0\du);}
%{\draw (25\du,0\du)--(33\du,0\du);}
{\draw (4.65\du,0.7\du)--(13.35\du,0.7\du);}
{\draw (4.65\du,-0.7\du)--(13.35\du,-0.7\du);}

%%%%%% ARROW HEAD

{\pgfsetcornersarced{\pgfpoint{0.300000\du}{0.300000\du}}\definecolor{dialinecolor}{rgb}{0.000000, 0.000000, 0.000000}
\pgfsetstrokecolor{dialinecolor}
\draw (7\du,-1.2\du)--(10.8\du,0\du)--(7\du,1.2\du);}

%%%%%% TAGS
%\node[anchor=west] at (28\du,0\du){${\rm F}_4$};

%\node[anchor=south] at (-6\du,1.1\du){$\scriptstyle 1$};
%
%\node[anchor=south] at (4\du,1.1\du){$\scriptstyle 2$};
%
%\node[anchor=south] at (14\du,1.1\du){$\scriptstyle 3$};
%
%\node[anchor=south] at (24\du,1.1\du){$\scriptstyle 4$};

\end{tikzpicture} 

%% file: DiagE66.tex
\ifx\du\undefined 
  \newlength{\du}
\fi
\setlength{\du}{3.3\unitlength}
\begin{tikzpicture}
\pgftransformxscale{1.000000}
\pgftransformyscale{1.000000}

%%%%%% COLORS
\definecolor{dialinecolor}{rgb}{0.000000, 0.000000, 0.000000} % EXTERIOR
\pgfsetstrokecolor{dialinecolor}
\definecolor{dialinecolor}{rgb}{0.000000, 0.000000, 0.000000} % INTERIOR
\pgfsetfillcolor{dialinecolor}

%%%%%% NODES

\pgfsetlinewidth{0.300000\du}
\pgfsetdash{}{0pt}
\pgfsetdash{}{0pt}
%\pgfsetmiterjoin

%%% #1
\pgfpathellipse{\pgfpoint{-6\du}{0\du}}{\pgfpoint{1\du}{0\du}}{\pgfpoint{0\du}{1\du}}
\pgfusepath{stroke}
\node at (-6\du,0\du){};
%\pgfpathellipse{\pgfpoint{-6\du}{0\du}}{\pgfpoint{1\du}{0\du}}{\pgfpoint{0\du}{1\du}}
%\pgfusepath{fill}
%\node at (-6\du,0\du){};

%%% #2
\pgfpathellipse{\pgfpoint{14\du}{-7\du}}{\pgfpoint{1\du}{0\du}}{\pgfpoint{0\du}{1\du}}
\pgfusepath{stroke}
\node at (14\du,-7\du){};
%\pgfpathellipse{\pgfpoint{14\du}{-7\du}}{\pgfpoint{1\du}{0\du}}{\pgfpoint{0\du}{1\du}}
%\pgfusepath{fill}
%\node at (14\du,-7\du){};

%%% #3
\pgfpathellipse{\pgfpoint{4\du}{0\du}}{\pgfpoint{1\du}{0\du}}{\pgfpoint{0\du}{1\du}}
\pgfusepath{stroke}
\node at (4\du,0\du){};
%\pgfpathellipse{\pgfpoint{4\du}{0\du}}{\pgfpoint{1\du}{0\du}}{\pgfpoint{0\du}{1\du}}
%\pgfusepath{fill}
%\node at (4\du,0\du){};

%%% #4
\pgfpathellipse{\pgfpoint{14\du}{0\du}}{\pgfpoint{1\du}{0\du}}{\pgfpoint{0\du}{1\du}}
\pgfusepath{stroke}
\node at (14\du,0\du){};
%\pgfpathellipse{\pgfpoint{14\du}{0\du}}{\pgfpoint{1\du}{0\du}}{\pgfpoint{0\du}{1\du}}
%\pgfusepath{fill}
%\node at (14\du,0\du){};

%%% #5
\pgfpathellipse{\pgfpoint{24\du}{0\du}}{\pgfpoint{1\du}{0\du}}{\pgfpoint{0\du}{1\du}}
\pgfusepath{stroke}
\node at (24\du,0\du){};
%\pgfpathellipse{\pgfpoint{24\du}{0\du}}{\pgfpoint{1\du}{0\du}}{\pgfpoint{0\du}{1\du}}
%\pgfusepath{fill}
%\node at (24\du,0\du){};

%%% #6
%\pgfpathellipse{\pgfpoint{34\du}{0\du}}{\pgfpoint{1\du}{0\du}}{\pgfpoint{0\du}{1\du}}
%\pgfusepath{stroke}
%\node at (34\du,0\du){};
\pgfpathellipse{\pgfpoint{34\du}{0\du}}{\pgfpoint{1\du}{0\du}}{\pgfpoint{0\du}{1\du}}
\pgfusepath{fill}
\node at (34\du,0\du){};

%%%%%% LINKS
\pgfsetlinewidth{0.300000\du}
\pgfsetdash{}{0pt}
\pgfsetdash{}{0pt}
\pgfsetbuttcap

{\draw (-5\du,0\du)--(3\du,0\du);}
{\draw (5\du,0\du)--(13\du,0\du);}
{\draw (15\du,0\du)--(23\du,0\du);}
{\draw (25\du,0\du)--(33\du,0\du);}
{\draw (14\du,-1\du)--(14\du,-6\du);}
%{\draw (34.65\du,0.7\du)--(43.35\du,0.7\du);}
%{\draw (34.65\du,-0.7\du)--(43.35\du,-0.7\du);}

%%%%%% TAGS
%\node[anchor=west] at (38\du,0\du){${\rm E}_6$};
%
%\node[anchor=south] at (-6\du,1.1\du){$\scriptstyle 1$};
%
%\node[anchor=south] at (4\du,1.1\du){$\scriptstyle 3$};
%
%\node[anchor=south] at (14\du,1.1\du){$\scriptstyle 4$};
%
%\node[anchor=south] at (24\du,1.1\du){$\scriptstyle 5$};
%
%\node[anchor=south] at (34\du,1.1\du){$\scriptstyle 6$};
%
%\node[anchor=east] at (12.9\du,-7\du){$\scriptstyle 2$};

\end{tikzpicture} 

%% file: DiagE66VMRT.tex
\ifx\du\undefined
  \newlength{\du}
\fi
\setlength{\du}{3.3\unitlength}
\begin{tikzpicture}
\pgftransformxscale{1.000000}
\pgftransformyscale{1.000000}

%%%%%% COLORS
\definecolor{dialinecolor}{rgb}{0.000000, 0.000000, 0.000000} % EXTERIOR
\pgfsetstrokecolor{dialinecolor}
\definecolor{dialinecolor}{rgb}{0.000000, 0.000000, 0.000000} % INTERIOR
\pgfsetfillcolor{dialinecolor}

%%%%%% NODES

\pgfsetlinewidth{0.300000\du}
\pgfsetdash{}{0pt}
\pgfsetdash{}{0pt}
%\pgfsetmiterjoin

%%%% #1
\pgfpathellipse{\pgfpoint{-6\du}{0\du}}{\pgfpoint{1\du}{0\du}}{\pgfpoint{0\du}{1\du}}
\pgfusepath{stroke}
\node at (-6\du,0\du){};
%\pgfpathellipse{\pgfpoint{-6\du}{0\du}}{\pgfpoint{1\du}{0\du}}{\pgfpoint{0\du}{1\du}}
%\pgfusepath{fill}
%\node at (-6\du,0\du){};

%%% #2
\pgfpathellipse{\pgfpoint{14\du}{-7\du}}{\pgfpoint{1\du}{0\du}}{\pgfpoint{0\du}{1\du}}
\pgfusepath{stroke}
\node at (14\du,-7\du){};
%\pgfpathellipse{\pgfpoint{14\du}{-7\du}}{\pgfpoint{1\du}{0\du}}{\pgfpoint{0\du}{1\du}}
%\pgfusepath{fill}
%\node at (14\du,-7\du){};

%%% #3
\pgfpathellipse{\pgfpoint{4\du}{0\du}}{\pgfpoint{1\du}{0\du}}{\pgfpoint{0\du}{1\du}}
\pgfusepath{stroke}
\node at (4\du,0\du){};
%\pgfpathellipse{\pgfpoint{4\du}{0\du}}{\pgfpoint{1\du}{0\du}}{\pgfpoint{0\du}{1\du}}
%\pgfusepath{fill}
%\node at (4\du,0\du){};

%%% #4
\pgfpathellipse{\pgfpoint{14\du}{0\du}}{\pgfpoint{1\du}{0\du}}{\pgfpoint{0\du}{1\du}}
\pgfusepath{stroke}
\node at (14\du,0\du){};
%\pgfpathellipse{\pgfpoint{14\du}{0\du}}{\pgfpoint{1\du}{0\du}}{\pgfpoint{0\du}{1\du}}
%\pgfusepath{fill}
%\node at (14\du,0\du){};

%%% #5
%\pgfpathellipse{\pgfpoint{24\du}{0\du}}{\pgfpoint{1\du}{0\du}}{\pgfpoint{0\du}{1\du}}
%\pgfusepath{stroke}
%\node at (24\du,0\du){};
\pgfpathellipse{\pgfpoint{24\du}{0\du}}{\pgfpoint{1\du}{0\du}}{\pgfpoint{0\du}{1\du}}
\pgfusepath{fill}
\node at (24\du,0\du){};

%%% #6
%\pgfpathellipse{\pgfpoint{34\du}{0\du}}{\pgfpoint{1\du}{0\du}}{\pgfpoint{0\du}{1\du}}
%\pgfusepath{stroke}
%\node at (34\du,0\du){};
%%\pgfpathellipse{\pgfpoint{34\du}{0\du}}{\pgfpoint{1\du}{0\du}}{\pgfpoint{0\du}{1\du}}
%%\pgfusepath{fill}
%%\node at (34\du,0\du){};

%%%%%% LINKS
\pgfsetlinewidth{0.300000\du}
\pgfsetdash{}{0pt}
\pgfsetdash{}{0pt}
\pgfsetbuttcap

{\draw (-5\du,0\du)--(3\du,0\du);}
{\draw (5\du,0\du)--(13\du,0\du);}
{\draw (15\du,0\du)--(23\du,0\du);}
%{\draw (25\du,0\du)--(33\du,0\du);}
{\draw (14\du,-1\du)--(14\du,-6\du);}
%{\draw (34.65\du,0.7\du)--(43.35\du,0.7\du);}
%{\draw (34.65\du,-0.7\du)--(43.35\du,-0.7\du);}

%%%%%% TAGS
%\node[anchor=west] at (38\du,0\du){${\rm E}_6$};
%
%\node[anchor=south] at (-6\du,1.1\du){$\scriptstyle 1$};
%
%\node[anchor=south] at (4\du,1.1\du){$\scriptstyle 3$};
%
%\node[anchor=south] at (14\du,1.1\du){$\scriptstyle 4$};
%
%\node[anchor=south] at (24\du,1.1\du){$\scriptstyle 5$};
%
%\node[anchor=south] at (34\du,1.1\du){$\scriptstyle 6$};
%
%\node[anchor=east] at (12.9\du,-7\du){$\scriptstyle 2$};

\end{tikzpicture} 

%% file: DiagD5num.tex
\ifx\du\undefined
  \newlength{\du}
\fi
\setlength{\du}{3.3\unitlength}
\begin{tikzpicture}
\pgftransformxscale{1.000000}
\pgftransformyscale{1.000000}

%%%%%% COLORS
\definecolor{dialinecolor}{rgb}{0.000000, 0.000000, 0.000000} % EXTERIOR
\pgfsetstrokecolor{dialinecolor}
\definecolor{dialinecolor}{rgb}{0.000000, 0.000000, 0.000000} % INTERIOR
\pgfsetfillcolor{dialinecolor}

%%%%%% NODES

\pgfsetlinewidth{0.300000\du}
\pgfsetdash{}{0pt}
\pgfsetdash{}{0pt}
%\pgfsetmiterjoin

%%% #1
\pgfpathellipse{\pgfpoint{-6\du}{0\du}}{\pgfpoint{1\du}{0\du}}{\pgfpoint{0\du}{1\du}}
\pgfusepath{stroke}
\node at (-6\du,0\du){};
%\pgfpathellipse{\pgfpoint{-6\du}{0\du}}{\pgfpoint{1\du}{0\du}}{\pgfpoint{0\du}{1\du}}
%\pgfusepath{fill}
%\node at (-6\du,0\du){};

%%% #2
\pgfpathellipse{\pgfpoint{4\du}{0\du}}{\pgfpoint{1\du}{0\du}}{\pgfpoint{0\du}{1\du}}
\pgfusepath{stroke}
\node at (4\du,0\du){};
%\pgfpathellipse{\pgfpoint{4\du}{0\du}}{\pgfpoint{1\du}{0\du}}{\pgfpoint{0\du}{1\du}}
%\pgfusepath{fill}
%\node at (4\du,0\du){};

%%% #3
\pgfpathellipse{\pgfpoint{14\du}{0\du}}{\pgfpoint{1\du}{0\du}}{\pgfpoint{0\du}{1\du}}
\pgfusepath{stroke}
\node at (14\du,0\du){};
%\pgfpathellipse{\pgfpoint{14\du}{0\du}}{\pgfpoint{1\du}{0\du}}{\pgfpoint{0\du}{1\du}}
%\pgfusepath{fill}
%\node at (14\du,0\du){};

%%% #4
\pgfpathellipse{\pgfpoint{20\du}{5\du}}{\pgfpoint{1\du}{0\du}}{\pgfpoint{0\du}{1\du}}
\pgfusepath{stroke}
\node at (14\du,0\du){};
%\pgfpathellipse{\pgfpoint{20\du}{5\du}}{\pgfpoint{1\du}{0\du}}{\pgfpoint{0\du}{1\du}}
%\pgfusepath{fill}
%\node at (14\du,0\du){};

%%% #5
\pgfpathellipse{\pgfpoint{20\du}{-5\du}}{\pgfpoint{1\du}{0\du}}{\pgfpoint{0\du}{1\du}}
\pgfusepath{stroke}
\node at (20\du,5\du){};
%\pgfpathellipse{\pgfpoint{20\du}{-5\du}}{\pgfpoint{1\du}{0\du}}{\pgfpoint{0\du}{1\du}}
%\pgfusepath{fill}
%\node at (20\du,-5\du){};

%%%%%% LINKS
\pgfsetlinewidth{0.300000\du}
\pgfsetdash{}{0pt}
\pgfsetdash{}{0pt}
\pgfsetbuttcap

{\draw (-5\du,0\du)--(3\du,0\du);}
{\draw (5\du,0\du)--(13\du,0\du);}
{\draw (14.65\du,0.7\du)--(19.13\du,4.5\du);}
{\draw (14.65\du,-0.7\du)--(19.13\du,-4.5\du);}

%%%%%% TAGS
%\node[anchor=west] at (24\du,0\du){${\rm D}_5$};

\node[anchor=south] at (-6\du,1.1\du){$\scriptstyle 1$};

\node[anchor=south] at (4\du,1.1\du){$\scriptstyle 2$};

\node[anchor=south] at (14\du,1.1\du){$\scriptstyle 3$};

\node[anchor=south] at (22\du,5\du){$\scriptstyle 4$};

\node[anchor=south] at (22\du,-5\du){$\scriptstyle 5$};

\end{tikzpicture} 

%% file: DiagB3tagged.tex
\ifx\du\undefined
  \newlength{\du}
\fi
\setlength{\du}{3.3\unitlength}
\begin{tikzpicture}
\pgftransformxscale{1.000000}
\pgftransformyscale{1.000000}

%%%%%% COLORS
\definecolor{dialinecolor}{rgb}{0.000000, 0.000000, 0.000000} % EXTERIOR
\pgfsetstrokecolor{dialinecolor}
\definecolor{dialinecolor}{rgb}{0.000000, 0.000000, 0.000000} % INTERIOR
\pgfsetfillcolor{dialinecolor}

%%%%%% NODES

\pgfsetlinewidth{0.300000\du}
\pgfsetdash{}{0pt}
\pgfsetdash{}{0pt}
%\pgfsetmiterjoin

%%% #1
\pgfpathellipse{\pgfpoint{-6\du}{0\du}}{\pgfpoint{1\du}{0\du}}{\pgfpoint{0\du}{1\du}}
\pgfusepath{stroke}
\node at (-6\du,0\du){};
%\pgfpathellipse{\pgfpoint{-6\du}{0\du}}{\pgfpoint{1\du}{0\du}}{\pgfpoint{0\du}{1\du}}
%\pgfusepath{fill}
%\node at (-6\du,0\du){};

%%% #2
\pgfpathellipse{\pgfpoint{4\du}{0\du}}{\pgfpoint{1\du}{0\du}}{\pgfpoint{0\du}{1\du}}
\pgfusepath{stroke}
\node at (4\du,0\du){};
%\pgfpathellipse{\pgfpoint{4\du}{0\du}}{\pgfpoint{1\du}{0\du}}{\pgfpoint{0\du}{1\du}}
%\pgfusepath{fill}
%\node at (4\du,0\du){};

%%% #3
\pgfpathellipse{\pgfpoint{14\du}{0\du}}{\pgfpoint{1\du}{0\du}}{\pgfpoint{0\du}{1\du}}
\pgfusepath{stroke}
\node at (14\du,0\du){};
%\pgfpathellipse{\pgfpoint{14\du}{0\du}}{\pgfpoint{1\du}{0\du}}{\pgfpoint{0\du}{1\du}}
%\pgfusepath{fill}
%\node at (14\du,0\du){};

%%%%%% LINKS
\pgfsetlinewidth{0.300000\du}
\pgfsetdash{}{0pt}
\pgfsetdash{}{0pt}
\pgfsetbuttcap

{\draw (-5\du,0\du)--(3\du,0\du);}
{\draw (4.65\du,0.7\du)--(13.35\du,0.7\du);}
{\draw (4.65\du,-0.7\du)--(13.35\du,-0.7\du);}

%%%%%% ARROW HEAD

{\pgfsetcornersarced{\pgfpoint{0.300000\du}{0.300000\du}}
%\definecolor{dialinecolor}{rgb}{1.000000, 1.000000, 1.000000}
%\pgfsetstrokecolor{dialinecolor}
\draw (7\du,-1.2\du)--(10.8\du,0\du)--(7\du,1.2\du);}

%%%%%% TAGS
%\node[anchor=west] at (18\du,0\du){${\rm B}_3$};

\node[anchor=south] at (-6\du,1.1\du){$\scriptstyle d_1$};

\node[anchor=south] at (4\du,1.1\du){$\scriptstyle d_2$};

\node[anchor=south] at (14\du,1.1\du){$\scriptstyle d_3$};

\end{tikzpicture} 

%% file: DiagA6num1.tex
\ifx\du\undefined
  \newlength{\du}
\fi
\setlength{\du}{3.3\unitlength}
\begin{tikzpicture}
\pgftransformxscale{1.000000}
\pgftransformyscale{1.000000}

%%%%%% COLORS
\definecolor{dialinecolor}{rgb}{0.000000, 0.000000, 0.000000} % EXTERIOR
\pgfsetstrokecolor{dialinecolor}
\definecolor{dialinecolor}{rgb}{0.000000, 0.000000, 0.000000} % INTERIOR
\pgfsetfillcolor{dialinecolor}

%%%%%% NODES

\pgfsetlinewidth{0.300000\du}
\pgfsetdash{}{0pt}
\pgfsetdash{}{0pt}
%\pgfsetmiterjoin

%%% #1
\pgfpathellipse{\pgfpoint{-6\du}{0\du}}{\pgfpoint{1\du}{0\du}}{\pgfpoint{0\du}{1\du}}
\pgfusepath{stroke}
\node at (-6\du,0\du){};
%\pgfpathellipse{\pgfpoint{-6\du}{0\du}}{\pgfpoint{1\du}{0\du}}{\pgfpoint{0\du}{1\du}}
%\pgfusepath{fill}
%\node at (-6\du,0\du){};

%%% #2
\pgfpathellipse{\pgfpoint{4\du}{0\du}}{\pgfpoint{1\du}{0\du}}{\pgfpoint{0\du}{1\du}}
\pgfusepath{stroke}
\node at (4\du,0\du){};
%\pgfpathellipse{\pgfpoint{4\du}{0\du}}{\pgfpoint{1\du}{0\du}}{\pgfpoint{0\du}{1\du}}
%\pgfusepath{fill}
%\node at (4\du,0\du){};

%%% #3
\pgfpathellipse{\pgfpoint{14\du}{0\du}}{\pgfpoint{1\du}{0\du}}{\pgfpoint{0\du}{1\du}}
\pgfusepath{stroke}
\node at (14\du,0\du){};
%\pgfpathellipse{\pgfpoint{14\du}{0\du}}{\pgfpoint{1\du}{0\du}}{\pgfpoint{0\du}{1\du}}
%\pgfusepath{fill}
%\node at (14\du,0\du){};

%%% #4
\pgfpathellipse{\pgfpoint{24\du}{0\du}}{\pgfpoint{1\du}{0\du}}{\pgfpoint{0\du}{1\du}}
\pgfusepath{stroke}
\node at (24\du,0\du){};
%\pgfpathellipse{\pgfpoint{24\du}{0\du}}{\pgfpoint{1\du}{0\du}}{\pgfpoint{0\du}{1\du}}
%\pgfusepath{fill}
%\node at (24\du,0\du){};

%%% #5
\pgfpathellipse{\pgfpoint{34\du}{0\du}}{\pgfpoint{1\du}{0\du}}{\pgfpoint{0\du}{1\du}}
\pgfusepath{stroke}
\node at (34\du,0\du){};
%\pgfpathellipse{\pgfpoint{34\du}{0\du}}{\pgfpoint{1\du}{0\du}}{\pgfpoint{0\du}{1\du}}
%\pgfusepath{fill}
%\node at (34\du,0\du){};

%%% #6
\pgfpathellipse{\pgfpoint{44\du}{0\du}}{\pgfpoint{1\du}{0\du}}{\pgfpoint{0\du}{1\du}}
\pgfusepath{stroke}
\node at (44\du,0\du){};
%\pgfpathellipse{\pgfpoint{14\du}{-7\du}}{\pgfpoint{1\du}{0\du}}{\pgfpoint{0\du}{1\du}}
%\pgfusepath{fill}
%\node at (14\du,-7\du){};

%%%%%% LINKS
\pgfsetlinewidth{0.300000\du}
\pgfsetdash{}{0pt}
\pgfsetdash{}{0pt}
\pgfsetbuttcap

{\draw (-5\du,0\du)--(3\du,0\du);}
{\draw (5\du,0\du)--(13\du,0\du);}
{\draw (15\du,0\du)--(23\du,0\du);}
{\draw (25\du,0\du)--(33\du,0\du);}
{\draw (35\du,0\du)--(43\du,0\du);}
%{\draw (34.65\du,0.7\du)--(43.35\du,0.7\du);}
%{\draw (34.65\du,-0.7\du)--(43.35\du,-0.7\du);}

%%%%%% TAGS
%\node[anchor=west] at (38\du,0\du){${\rm E}_6$};

\node[anchor=south] at (-6\du,1.1\du){$\scriptstyle d_1$};

\node[anchor=south] at (4\du,1.1\du){$\scriptstyle d_2$};

\node[anchor=south] at (14\du,1.1\du){$\scriptstyle d_3$};

\node[anchor=south] at (24\du,1.1\du){$\scriptstyle d_3$};

\node[anchor=south] at (34\du,1.1\du){$\scriptstyle d_2$};

\node[anchor=south] at (44\du,1.1\du){$\scriptstyle d_1$};

\end{tikzpicture} 

%% file: DiagA7num1.tex
\ifx\du\undefined
  \newlength{\du}
\fi
\setlength{\du}{3.3\unitlength}
\begin{tikzpicture}
\pgftransformxscale{1.000000}
\pgftransformyscale{1.000000}

%%%%%% COLORS
\definecolor{dialinecolor}{rgb}{0.000000, 0.000000, 0.000000} % EXTERIOR
\pgfsetstrokecolor{dialinecolor}
\definecolor{dialinecolor}{rgb}{0.000000, 0.000000, 0.000000} % INTERIOR
\pgfsetfillcolor{dialinecolor}

%%%%%% NODES

\pgfsetlinewidth{0.300000\du}
\pgfsetdash{}{0pt}
\pgfsetdash{}{0pt}
%\pgfsetmiterjoin

%%% #1
\pgfpathellipse{\pgfpoint{-6\du}{0\du}}{\pgfpoint{1\du}{0\du}}{\pgfpoint{0\du}{1\du}}
\pgfusepath{stroke}
\node at (-6\du,0\du){};
%\pgfpathellipse{\pgfpoint{-6\du}{0\du}}{\pgfpoint{1\du}{0\du}}{\pgfpoint{0\du}{1\du}}
%\pgfusepath{fill}
%\node at (-6\du,0\du){};

%%% #2
\pgfpathellipse{\pgfpoint{4\du}{0\du}}{\pgfpoint{1\du}{0\du}}{\pgfpoint{0\du}{1\du}}
\pgfusepath{stroke}
\node at (4\du,0\du){};
%\pgfpathellipse{\pgfpoint{4\du}{0\du}}{\pgfpoint{1\du}{0\du}}{\pgfpoint{0\du}{1\du}}
%\pgfusepath{fill}
%\node at (4\du,0\du){};

%%% #3
\pgfpathellipse{\pgfpoint{14\du}{0\du}}{\pgfpoint{1\du}{0\du}}{\pgfpoint{0\du}{1\du}}
\pgfusepath{stroke}
\node at (14\du,0\du){};
%\pgfpathellipse{\pgfpoint{14\du}{0\du}}{\pgfpoint{1\du}{0\du}}{\pgfpoint{0\du}{1\du}}
%\pgfusepath{fill}
%\node at (14\du,0\du){};

%%% #4
\pgfpathellipse{\pgfpoint{24\du}{0\du}}{\pgfpoint{1\du}{0\du}}{\pgfpoint{0\du}{1\du}}
\pgfusepath{stroke}
\node at (24\du,0\du){};
%\pgfpathellipse{\pgfpoint{24\du}{0\du}}{\pgfpoint{1\du}{0\du}}{\pgfpoint{0\du}{1\du}}
%\pgfusepath{fill}
%\node at (24\du,0\du){};

%%% #5
\pgfpathellipse{\pgfpoint{34\du}{0\du}}{\pgfpoint{1\du}{0\du}}{\pgfpoint{0\du}{1\du}}
\pgfusepath{stroke}
\node at (34\du,0\du){};
%\pgfpathellipse{\pgfpoint{34\du}{0\du}}{\pgfpoint{1\du}{0\du}}{\pgfpoint{0\du}{1\du}}
%\pgfusepath{fill}
%\node at (34\du,0\du){};

%%% #6
\pgfpathellipse{\pgfpoint{44\du}{0\du}}{\pgfpoint{1\du}{0\du}}{\pgfpoint{0\du}{1\du}}
\pgfusepath{stroke}
\node at (44\du,0\du){};
%\pgfpathellipse{\pgfpoint{14\du}{-7\du}}{\pgfpoint{1\du}{0\du}}{\pgfpoint{0\du}{1\du}}
%\pgfusepath{fill}
%\node at (14\du,-7\du){};

%%% #7
\pgfpathellipse{\pgfpoint{54\du}{0\du}}{\pgfpoint{1\du}{0\du}}{\pgfpoint{0\du}{1\du}}
\pgfusepath{stroke}
\node at (54\du,0\du){};
%\pgfpathellipse{\pgfpoint{14\du}{-7\du}}{\pgfpoint{1\du}{0\du}}{\pgfpoint{0\du}{1\du}}
%\pgfusepath{fill}
%\node at (14\du,-7\du){};

%%%%%% LINKS
\pgfsetlinewidth{0.300000\du}
\pgfsetdash{}{0pt}
\pgfsetdash{}{0pt}
\pgfsetbuttcap

{\draw (-5\du,0\du)--(3\du,0\du);}
{\draw (5\du,0\du)--(13\du,0\du);}
{\draw (15\du,0\du)--(23\du,0\du);}
{\draw (25\du,0\du)--(33\du,0\du);}
{\draw (35\du,0\du)--(43\du,0\du);}
{\draw (45\du,0\du)--(53\du,0\du);}
%{\draw (34.65\du,0.7\du)--(43.35\du,0.7\du);}
%{\draw (34.65\du,-0.7\du)--(43.35\du,-0.7\du);}

%%%%%% TAGS
%\node[anchor=west] at (38\du,0\du){${\rm E}_6$};

\node[anchor=south] at (-6\du,1.1\du){$\scriptstyle d_3$};

\node[anchor=south] at (4\du,1.1\du){$\scriptstyle d_2$};

\node[anchor=south] at (14\du,1.1\du){$\scriptstyle d_1$};

\node[anchor=south] at (24\du,1.1\du){$\scriptstyle d_3$};

\node[anchor=south] at (34\du,1.1\du){$\scriptstyle d_1$};

\node[anchor=south] at (44\du,1.1\du){$\scriptstyle d_2$};

\node[anchor=south] at (54\du,1.1\du){$\scriptstyle d_3$};

\end{tikzpicture} 

%% file: DiagA7num2.tex
\ifx\du\undefined
  \newlength{\du}
\fi
\setlength{\du}{3.3\unitlength}
\begin{tikzpicture}
\pgftransformxscale{1.000000}
\pgftransformyscale{1.000000}

%%%%%% COLORS
\definecolor{dialinecolor}{rgb}{0.000000, 0.000000, 0.000000} % EXTERIOR
\pgfsetstrokecolor{dialinecolor}
\definecolor{dialinecolor}{rgb}{0.000000, 0.000000, 0.000000} % INTERIOR
\pgfsetfillcolor{dialinecolor}

%%%%%% NODES

\pgfsetlinewidth{0.300000\du}
\pgfsetdash{}{0pt}
\pgfsetdash{}{0pt}
%\pgfsetmiterjoin

%%% #1
\pgfpathellipse{\pgfpoint{-6\du}{0\du}}{\pgfpoint{1\du}{0\du}}{\pgfpoint{0\du}{1\du}}
\pgfusepath{stroke}
\node at (-6\du,0\du){};
%\pgfpathellipse{\pgfpoint{-6\du}{0\du}}{\pgfpoint{1\du}{0\du}}{\pgfpoint{0\du}{1\du}}
%\pgfusepath{fill}
%\node at (-6\du,0\du){};

%%% #2
\pgfpathellipse{\pgfpoint{4\du}{0\du}}{\pgfpoint{1\du}{0\du}}{\pgfpoint{0\du}{1\du}}
\pgfusepath{stroke}
\node at (4\du,0\du){};
%\pgfpathellipse{\pgfpoint{4\du}{0\du}}{\pgfpoint{1\du}{0\du}}{\pgfpoint{0\du}{1\du}}
%\pgfusepath{fill}
%\node at (4\du,0\du){};

%%% #3
\pgfpathellipse{\pgfpoint{14\du}{0\du}}{\pgfpoint{1\du}{0\du}}{\pgfpoint{0\du}{1\du}}
\pgfusepath{stroke}
\node at (14\du,0\du){};
%\pgfpathellipse{\pgfpoint{14\du}{0\du}}{\pgfpoint{1\du}{0\du}}{\pgfpoint{0\du}{1\du}}
%\pgfusepath{fill}
%\node at (14\du,0\du){};

%%% #4
\pgfpathellipse{\pgfpoint{24\du}{0\du}}{\pgfpoint{1\du}{0\du}}{\pgfpoint{0\du}{1\du}}
\pgfusepath{stroke}
\node at (24\du,0\du){};
%\pgfpathellipse{\pgfpoint{24\du}{0\du}}{\pgfpoint{1\du}{0\du}}{\pgfpoint{0\du}{1\du}}
%\pgfusepath{fill}
%\node at (24\du,0\du){};

%%% #5
\pgfpathellipse{\pgfpoint{34\du}{0\du}}{\pgfpoint{1\du}{0\du}}{\pgfpoint{0\du}{1\du}}
\pgfusepath{stroke}
\node at (34\du,0\du){};
%\pgfpathellipse{\pgfpoint{34\du}{0\du}}{\pgfpoint{1\du}{0\du}}{\pgfpoint{0\du}{1\du}}
%\pgfusepath{fill}
%\node at (34\du,0\du){};

%%% #6
\pgfpathellipse{\pgfpoint{44\du}{0\du}}{\pgfpoint{1\du}{0\du}}{\pgfpoint{0\du}{1\du}}
\pgfusepath{stroke}
\node at (44\du,0\du){};
%\pgfpathellipse{\pgfpoint{14\du}{-7\du}}{\pgfpoint{1\du}{0\du}}{\pgfpoint{0\du}{1\du}}
%\pgfusepath{fill}
%\node at (14\du,-7\du){};

%%% #7
\pgfpathellipse{\pgfpoint{54\du}{0\du}}{\pgfpoint{1\du}{0\du}}{\pgfpoint{0\du}{1\du}}
\pgfusepath{stroke}
\node at (54\du,0\du){};
%\pgfpathellipse{\pgfpoint{14\du}{-7\du}}{\pgfpoint{1\du}{0\du}}{\pgfpoint{0\du}{1\du}}
%\pgfusepath{fill}
%\node at (14\du,-7\du){};

%%%%%% LINKS
\pgfsetlinewidth{0.300000\du}
\pgfsetdash{}{0pt}
\pgfsetdash{}{0pt}
\pgfsetbuttcap

{\draw (-5\du,0\du)--(3\du,0\du);}
{\draw (5\du,0\du)--(13\du,0\du);}
{\draw (15\du,0\du)--(23\du,0\du);}
{\draw (25\du,0\du)--(33\du,0\du);}
{\draw (35\du,0\du)--(43\du,0\du);}
{\draw (45\du,0\du)--(53\du,0\du);}
%{\draw (34.65\du,0.7\du)--(43.35\du,0.7\du);}
%{\draw (34.65\du,-0.7\du)--(43.35\du,-0.7\du);}

%%%%%% TAGS
%\node[anchor=west] at (38\du,0\du){${\rm E}_6$};

\node[anchor=south] at (-6\du,1.1\du){$\scriptstyle d_3$};

\node[anchor=south] at (4\du,1.1\du){$\scriptstyle d_2$};

\node[anchor=south] at (14\du,1.1\du){$\scriptstyle d_3$};

\node[anchor=south] at (24\du,1.1\du){$\scriptstyle d_1$};

\node[anchor=south] at (34\du,1.1\du){$\scriptstyle d_3$};

\node[anchor=south] at (44\du,1.1\du){$\scriptstyle d_2$};

\node[anchor=south] at (54\du,1.1\du){$\scriptstyle d_3$};

\end{tikzpicture} 

%% file: DiagB3tagged2.tex
\ifx\du\undefined
  \newlength{\du}
\fi
\setlength{\du}{3.3\unitlength}
\begin{tikzpicture}
\pgftransformxscale{1.000000}
\pgftransformyscale{1.000000}

%%%%%% COLORS
\definecolor{dialinecolor}{rgb}{0.000000, 0.000000, 0.000000} % EXTERIOR
\pgfsetstrokecolor{dialinecolor}
\definecolor{dialinecolor}{rgb}{0.000000, 0.000000, 0.000000} % INTERIOR
\pgfsetfillcolor{dialinecolor}

%%%%%% NODES

\pgfsetlinewidth{0.300000\du}
\pgfsetdash{}{0pt}
\pgfsetdash{}{0pt}
%\pgfsetmiterjoin

%%% #1
\pgfpathellipse{\pgfpoint{-6\du}{0\du}}{\pgfpoint{1\du}{0\du}}{\pgfpoint{0\du}{1\du}}
\pgfusepath{stroke}
\node at (-6\du,0\du){};
%\pgfpathellipse{\pgfpoint{-6\du}{0\du}}{\pgfpoint{1\du}{0\du}}{\pgfpoint{0\du}{1\du}}
%\pgfusepath{fill}
%\node at (-6\du,0\du){};

%%% #2
\pgfpathellipse{\pgfpoint{4\du}{0\du}}{\pgfpoint{1\du}{0\du}}{\pgfpoint{0\du}{1\du}}
\pgfusepath{stroke}
\node at (4\du,0\du){};
%\pgfpathellipse{\pgfpoint{4\du}{0\du}}{\pgfpoint{1\du}{0\du}}{\pgfpoint{0\du}{1\du}}
%\pgfusepath{fill}
%\node at (4\du,0\du){};

%%% #3
\pgfpathellipse{\pgfpoint{14\du}{0\du}}{\pgfpoint{1\du}{0\du}}{\pgfpoint{0\du}{1\du}}
\pgfusepath{stroke}
\node at (14\du,0\du){};
%\pgfpathellipse{\pgfpoint{14\du}{0\du}}{\pgfpoint{1\du}{0\du}}{\pgfpoint{0\du}{1\du}}
%\pgfusepath{fill}
%\node at (14\du,0\du){};

%%%%%% LINKS
\pgfsetlinewidth{0.300000\du}
\pgfsetdash{}{0pt}
\pgfsetdash{}{0pt}
\pgfsetbuttcap

{\draw (-5\du,0\du)--(3\du,0\du);}
{\draw (4.65\du,0.7\du)--(13.35\du,0.7\du);}
{\draw (4.65\du,-0.7\du)--(13.35\du,-0.7\du);}

%%%%%% ARROW HEAD

{\pgfsetcornersarced{\pgfpoint{0.300000\du}{0.300000\du}}
%\definecolor{dialinecolor}{rgb}{1.000000, 1.000000, 1.000000}
%\pgfsetstrokecolor{dialinecolor}
\draw (7\du,-1.2\du)--(10.8\du,0\du)--(7\du,1.2\du);}

%%%%%% TAGS
%\node[anchor=west] at (18\du,0\du){${\rm B}_3$};

\node[anchor=south] at (-6\du,1.1\du){$\scriptstyle 0$};

\node[anchor=south] at (4\du,1.1\du){$\scriptstyle 0$};

\node[anchor=south] at (14\du,1.1\du){$\scriptstyle 1$};

\end{tikzpicture} 